\documentclass[11pt]{article}
\usepackage{amsmath,amssymb,amsthm,color,comment,geometry,graphics,hyperref}
\usepackage[T1]{fontenc}
\usepackage[utf8]{inputenc}
\usepackage{authblk}

\newtheorem{theorem}{Theorem}[section]
\newtheorem{lemma}[theorem]{Lemma}
\newtheorem{proposition}[theorem]{Proposition}
\newtheorem{corollary}[theorem]{Corollary}
\theoremstyle{definition}
\newtheorem{definition}{Definition}[section]
\theoremstyle{remark}
\newtheorem{remark}{Remark}[section]
\newtheorem{example}{Example}[section]

\newcommand{\dom}{{\rm Dom}}

\newcommand{\Tr}{{\rm Tr}}
\newcommand{\tr}{{\rm tr}}

\newcommand{\scalar}{R}

\newcommand{\nctorus}[1][2]{\mathbb{T}_\theta^{#1}}
\newcommand{\A}[1][2]{C(\mathbb{T}_\theta^{#1})} 
\newcommand{\Ai}[1][2]{C^\infty(\mathbb{T}_\theta^{#1})}

\newcommand{\lap}{\triangle}

\makeatletter
\newcommand{\extp}{\@ifnextchar^\@extp{\@extp^{\,}}}
\def\@extp^#1{\mathop{\bigwedge\nolimits^{\!#1}}}
\makeatother
\renewcommand{\vec}[1]{\boldsymbol{#1}}
\renewcommand{\dim}{d}
\usepackage{tikz}

\title{Spectral geometry of functional metrics on noncommutative tori}

\author[$\dag$]{Asghar Ghorbanpour}
\author[$\dag$]{Masoud Khalkhali }
\affil[$\dag$]{Department of Mathematics, University of Western Ontario}

\newcommand{\Addresses}{{
  \bigskip
  \footnotesize

  Asghar Ghorbanpour,\textsc{Department of Mathematics, University of Western Ontario,
London, Ontario, Canada, N6A 5B7 }\par\nopagebreak
  \textit{E-mail address}:\texttt{aghorba@uwo.ca}

  \medskip

 Masoud Khalkhali, \textsc{Department of Mathematics, University of Western Ontario,
London, Ontario, Canada, N6A 5B7 }\par\nopagebreak
  \textit{E-mail address}: \texttt{masoud@uwo.ca}

}}

\date{ }
\begin{document}
\maketitle
\begin{abstract}
We introduce a new family of metrics, called functional metrics, on noncommutative tori and study their spectral geometry.
We define a class of  Laplace type operators for these metrics and study their  spectral invariants obtained from the heat trace asymptotics. 
A formula  for the second density of  the heat trace is obtained.
In particular, the scalar curvature density and the total scalar curvature of  functional metrics are explicitly computed    in all dimensions for   certain classes of metrics including  conformally flat metrics and twisted product of flat metrics.
Finally a Gauss-Bonnet type theorem for a noncommutative two torus equipped  with a general functional metric is proved. 
\end{abstract}

\tableofcontents{}
\allowdisplaybreaks

\section{Introduction}
Investigating the differential geometry of curved noncommutative tori started in 
\cite{Connes-Tretkoff2011, Fathizadeh-Khalkhali2012, Connes-Moscovici2014, Fathizadeh-Khalkhali2013} and was followed up in many publications. An incomplete  list includes   \cite{Fathizadeh2015,Fathizadeh-Khalkhali2015,Tanvir-Marcolli2012, Dbrowski-Sitarz2015, Lesch-Moscovici2016, Floricel-Ghorbanpour-Khalkhali2016,Dong-Ghorbanpour-Khalkhali2018}.
Unlike the Riemannian geometry of manifolds, which is defined by a metric tensor, the geometries of the noncommutative tori, or  noncommutative spaces, are defined  through operators which imitate the properties of geometric  operators such as Dirac or Laplace operators  \cite{Connes1994}.
On noncommutative tori such operators can be constructed by analogy using basic facts from the complex geometry or the theory of de Rham complexes on  smooth manifolds. 
Inspired by spectral geometry, some of the spectral invariants  of these operators can be used to define geometric quantities such as scalar curvature, total scalar curvature and even Ricci curvature \cite{Floricel-Ghorbanpour-Khalkhali2016,Dong-Ghorbanpour-Khalkhali2018}.
Similar ideas were worked out for noncommutative toric manifolds in \cite{Liu2017-I,Liu2018}.
Tools such as Connes' pseudo-differential calculus for C$^\ast$-dynamical systems \cite{Connes1980} and the rearrangement lemma \cite{Connes1980,Lesch2017, Connes-Tretkoff2011, Connes-Moscovici2014} play an essential role.   
 
 The geometries studied in the above  works  are mostly  restricted to conformally flat metrics.
The rationale behind the popularity of this specific class comes from the fact that by the uniformization theorem  the conformally flat metrics cover all possible geometries on two tori.
This is not true for the higher dimensional tori and for the two dimensional noncommutative torus the situation is much less understood as far  as the uniformization theorem goes.
Although the study of  conformally flat  geometries alone have shown many  interesting new phenomenon,  
we believe that this is only the tip of the iceberg and much more remains to be discovered for non-conformal metrics and in higher dimensions.

In this paper we introduce a new class of metrics, called {\it functional metrics}.
Let $\Ai[\dim]$ denotes the algebra of smooth elements of a $\dim$-dimensional nonocommutative torus. 
Fix a selfadjoint element $h\in \Ai[\dim]$ and let the functions $g_{ij}:\mathbb{R}\to\mathbb{R}$, $1\leq i,j\leq \dim$, be such that $(g_{ij}(t))$ is a positive definite symmetric matrix for all $t$ in the spectrum of $h$. 
If $\theta=0$, one can consider the metric tensor on $\nctorus[\dim]=(\mathbb{R}/(2\pi\mathbb{Z}))^\dim$ given by 
\begin{equation*}
g_{ij}(h)dx^idx^j.
\end{equation*}
Although, the metric tensor is not a well defined notion in the noncommutative case,  we can exploit the perspective explained in the earlier lines: find the geometric invariants using the spectral analysis of a good differential operator on the noncommutative tori.    
We construct an operator $\lap_{0,g}$ which is antiunitary equivalent to the Laplacian on functions of $\nctorus[\dim]$. 
To define the latter Laplacian, following the same lines as in \cite{Connes-Tretkoff2011}, the entries $g_{ij}(h)$ are used to deform the inner product on functions and 1-forms.

Having a candidate for a geometric operator in hand, we go ahead and compute the densities of the heat trace asymptotic expansion $\Tr(e^{-t\lap_{0,g}})$  using Connes' pseudodifferential calculus on the noncommutative tori.
It turns out that to handle this general case, we need to upgrade some of our tools.
At this stage, we found the point of view taken in \cite{Lesch2017}  very useful.
The primary goal of this work was to offer a different proof for the rearrangement lemma and it provided a rigorous way  to interpret the outcome elements of the original rearrangement lemma. 
We call this type of elements as written in the {\it contraction form}:
\begin{equation*}
F(h_{(0)},\cdots,h_{(n)})(b_1\cdot b_2\cdots b_n).
\end{equation*} 
By a slight variation in one of the main results  of \cite{Lesch2017}, we become well-equipped to handle the difficulties. 
We also noticed that the symbol calculus can be applied for  differential operators whose symbols can be written in the contraction form.
We call  these operators {\it $h$-differential operators}. 
We first find the way in which derivations act on elements written in the contraction form (Theorem \ref{derivationoffunctionofh}). 
The Newton divided difference calculus plays an important role.
For instance, we have
\begin{eqnarray*}
&\delta_j\big(f(h_{(0)},h_{(1)})(b_1)\big)=&\\
&f(h_{(0)},h_{(1)})(\delta_j(b_1))
+\big[h_{(0)},h_{(1)};f(\cdot, h_{(2)} )(\delta_j(h)\cdot b)\big]
+\big[h_{(1)},h_{(2)};f(h_{(0)},\cdot)(b\cdot \delta_j(h))\big].&
\end{eqnarray*}
Using these facts and applying the pseudodifferential calculus,  we compute the spectral densities for positive $h$-differential operators whose principal symbol is given by a functional metric.
We call such an operator a {\it Laplace type $h$-differential operator}.

This change in order of the computations, i.e. writing symbols in the contraction form from the beginning,  not only induced a smoother computation, at least symbolically, but also assisted us in a very fundamental way to consider general cases, e.g. finding the curvature in all higher dimensions for conformally flat and twisted product of flat metrics. 
Using the contraction form also enabled us to use Einstein summation convention. 

While computing the heat trace densities of a positive Laplace type $h$-differential operator $P$ with  principal symbol $P_2^{ij}(h)\xi_i\xi_j$, we come across integrals of the form 
\begin{equation*}
\frac{-1}{\pi^{\dim/2}}\int_{\mathbb{R}^{\dim}}\xi_{n_1}\cdots \xi_{n_{2|\vec{\alpha}|-4}}\frac{1}{2\pi i}\int_\gamma e^{-\lambda} B_0^{\alpha_0}(t_0)\cdots B_0^{\alpha_n}(t_n)  d\lambda d\xi,
\end{equation*}
where $B_0(t)=(P_2^{ij}(t)\xi_i\xi_j-\lambda)^{-1}$.
Such an integral is a smooth function of $t_j$'s and we call it a {\it $T$-function} and denote it by   $T_{\vec{n};\vec{\alpha}}(t_0,\cdots,t_n)$.
In Lemma \ref{integralformofTna}, we find a formula for $T$-functions as an integral over the standard simplex which helps us to evaluate these functions in some cases (see Lemmas \ref{Tfunctionsconfomallyflatlemma}, \ref{Talphaklemma} and  \ref{Tfunctionsdimtwo}).
As an example, $T$-functions for conformally flat principal symbols $P_2^{ij}(t)=f(t)g^{ij}$ for $\dim\neq 2$, are given by
\begin{equation*}
T_{\vec{n,\alpha}}(t_0,\cdots,t_n)=\sqrt{|g|}\underset{\vec{n}}{\sum \prod}g_{n_in_{\sigma(i)}} 
\frac{(-1)^{|\vec{\alpha}|-1}\Gamma(\frac{\dim}2-1)}{\Gamma(\frac{\dim}2+|\vec{\alpha}|-2)} \left.\partial_x^{\vec{\beta}} \big[x_0,\cdots,x_n;u^{1-\frac{d}2}\big]\right|_{x_j=f(t_j)},
\end{equation*}
where $\vec{\beta}=(\alpha_0-1,\cdots,\alpha_{n}-1)$.

In Theorem \ref{scalarcrvatureconformal}, we show that the scalar curvature of the $\dim$-dimensional noncommutative tori $\mathbb{T}_\theta^\dim$ equipped with a conformally flat metric $f(h)^{-1}g_{ij}$, where $g_{ij}$ are constant numbers, is of the form
\begin{equation*}
R=\sqrt{|g|} \Big(K_{\dim}(h_{(0)},h_{(1)})(g^{ij}\delta_i\delta_j(h))+H_{\dim}(h_{(0)},h_{(1)},h_{(2)})(g^{ij}\delta_i(h)\cdot\delta_j(h))\Big),
\end{equation*}
A formula for the functions $K_\dim$ and $H_\dim$ is given for all $\dim$ in Theorem \ref{scalarcrvatureconformal}.
In particular, for $\dim\neq 2$, the function $K_\dim$ is given by
\begin{align*}
K_{\dim}(t_0,t_1)&=\frac{4\, f(t_0)^{2-\frac{3 \dim }{4}} f(t_1)^{2-\frac{3 \dim }{4}} }{\dim(\dim -2) (f(t_0)-f(t_1))^3}\big[t_0,t_1;f\big]\times\\
&\left((\dim -1) f(t_0)^{\frac\dim 2} f(t_1)^{\frac{\dim }{2}-1}-(\dim -1) f(t_0)^{\frac{\dim }{2}-1} f(t_1)^{\frac \dim 2}-f(t_0)^{\dim -1}+f(t_1)^{\dim -1}\right).
\end{align*}

In Section \ref{twistedmetricsec}, we study the geometry of  the noncommutative tori $\nctorus[\dim]$ equipped with metrics of the form 
\begin{equation*}
f(t)^{-1}g\oplus \tilde{g}.
\end{equation*}
Following the differential geometry convention, we call such a metric a {\it twisted product functional metric} of $(\nctorus[r],g)$ and $(\nctorus[\dim-r],\tilde{g})$ with the twisting element $f(h)^{-1}$. 
The scalar curvature density of this metric is then obtained as
\begin{equation*}
\begin{aligned}
R=\sqrt{|g||\tilde{g}|} \Big(& K_{r}(h_{(0)},h_{(1)})(g^{ij}\delta_i\delta_j(h))+H_{r}(h_{(0)},h_{(1)},h_{(2)})(g^{ij}\delta_i(h)\cdot\delta_j(h))\\
&+\tilde{K}_{r}(h_{(0)},h_{(1)})(\tilde{g}^{ij}\delta_i\delta_j(h))+\tilde{H}_{r}(h_{(0)},h_{(1)},h_{(2)})(\tilde{g}^{ij}\delta_i(h)\cdot\delta_j(h))\Big).
\end{aligned}
\end{equation*} 
The new functions $\tilde{K}_{r}$ and $\tilde{H}_{r}$ for $r\neq 2, 4$ are  given in Theorem \ref{scalartwisted}.
For instance, we have 
\begin{align*}
\tilde{K}_{r}(t_0,t_1)\hspace*{-1mm}=\hspace*{-1mm}\frac{(2r-4)(f(t_0)^2-f(t_1)^2) f(t_0)^{\frac{r}2} f(t_1)^{\frac{r}2}+4 f(t_0)^2 f(t_1)^r-4 f(t_0)^r f(t_1)^2 }{(r-4) (r-2)f(t_0)^{\frac34 r} f(t_1)^{\frac34 r} (f(t_0)-f(t_1))^3}\hspace*{-0.5mm}\big[t_0,t_1;f\big].
\end{align*}

Finally the {\it total scalar curvature}, $\varphi(b_2(\lap_{0,g}))$, of functional metrics on noncommutative tori is studied  in Section \ref{Totalcurvaturesection}.
Note that the presence of the trace $\varphi$ provides a ground to simplify the formula for the total scalar curvature. 
One then has Proposition \ref{totalcurvature} which asserts that 
\begin{equation*}
\varphi(R)=\varphi\left(F_S^{ij}(h_{(0)},h_{(1)})\big(\delta_i(h)\big)\delta_j(h)\right).
\end{equation*}
The functions $F_S^{ij}$ are given by
\begin{equation*}
\begin{aligned}
F_S^{ij}(t_0,t_1)= \frac1{2(t_0-t_1)^2}\Big(
&A^{ij}\sqrt[4]{|g|(t_0)|g|(t_1)}-2A^{ij} T_{;1,1}(t_0,t_1)\\
&+T_{k,l;1,2}(t_0,t_1) \big(2  A^{ik} g^{lj}(t_1)+2A^{kj}  g^{il}(t_1)-A^{ik}A^{lj}\big)\\
&+T_{k,l;2,1}(t_0,t_1)\big (2A^{ik}g^{lj}(t_0)+2  A^{kj} g^{il}(t_0) -A^{ik}A^{lj}\big)\Big),
\end{aligned}
\end{equation*} 
where
$A^{ij}=
{|g|^{\frac14}(t_0 )}{|g|^{\frac{-1}4} (t_1 )}g^{i j} (t_0 )
+{|g|^{\frac14} (t_1 )}{|g|^{\frac{-1}4}(t_0)} g^{ij} (t_1 ).$
This formula is then used in Example \ref{totalcuroftwistedexample} to generate the functions for the total scalar curvature  for twisted product of flat metrics on $\nctorus[\dim]$.
As an another type of example, we compute  functions $F_{S}^{ij}$ for doubly twisted product metrics on $\nctorus[4]$ in Example \ref{dtflatT4}.
On the other hand, using the Proposition \ref{totalcurvature} and evaluating the functions $F_S^{ij}$, we prove a Gauss-Bonnet theorem in dimension two, which generalizes the results obtained for conformally flat noncommutative two torus in \cite{Connes-Tretkoff2011,Fathizadeh-Khalkhali2012}.
More precisely, Theorem \ref{GBindim2} states that the total scalar curvature $\varphi(R)$ for any functional  metrics $g$ on $\nctorus$ vanishes.

This paper is organized as follows.  
In Section 2, we review preliminaries such as noncommutative tori, Connes' pseudodifferential calculus and Newton divided difference calculus.
Also in Subsection 2.2 a version of the rearrangement lemma is provided and  used to find the derivation of an element written in the contraction form.
Section 3 deals with the heat trace asymptotics for Laplace type $h$-differential operators.
The $T$-functions are introduced in this section and their properties are studied.
The class of functional metrics on noncommutative tori is introduced in Section 4 and the scalar curvature of some of the examples of such geometries is explicitly computed in all dimensions.
In Section 5, we manage to track down further simplifications for the trace of the total scalar curvature for the functional metrics and prove a Gauss-Bonnet theorem for these metrics. 

\noindent{\bf Acknowledgement:}
The work of M. Khalkhali was partially supported by an NSERC Discovery Grant.
A. Ghorbanpour was partially supported by a Fields Institute postdoctoral fellowship, held at University of Western Ontario.

\section{Preliminaries and computational tools}
In \cite{Connes-Tretkoff2011,Fathizadeh-Khalkhali2012} to prove a version of the Gauss-Bonnet theorem for the noncommutative two torus and later in \cite{Connes-Moscovici2014,Fathizadeh-Khalkhali2013} to compute the scalar curvature of a curved noncommutative two torus equipped with a conformally flat metric, the second density of the heat trace of the Laplacian $D^2$ of the Dirac operator was studied. 
This term  was explicitly computed for the heat trace and the localized heat trace and the final results were  written in a novel form which we will call the contraction form. 
These computations involved two main steps. 
First, the  symbol of the  parametrix of the differential operator  were computed and several integrations were performed. 
The rearrangement lemma (cf. \cite{Connes-Moscovici2014}) plays  a central role in this step and it is the part that generates elements in the contraction form. 
In the second step, the outcome from the previous step, which are in terms of an auxiliary element $k=e^h$, were manipulated to be written in terms of the dilaton $h$ and then simplified to its best form.
One of our goals is to show that we do not need to wait until the end of the step one to have elements in the contraction from; 
we can start off with operators whose symbol is originally written in the contraction form.
In this section, in addition to recalling basics of the theory, we shall develop the required tools to be able to perform the symbol calculus on symbols in the contraction form.

\subsection{Noncommutative \texorpdfstring{$\dim$-tori}{dtori} and Connes' pseudodifferential calculus}
In this section, we review  the definitions and fix notations for the noncommutative $\dim$-dimensional torus. We also recall a  pseudodifferential calculus on these $C^\ast$-algebras which is a special case of the calculus introduced by Connes \cite{Connes1980} for $C^\ast$-dynamical systems.

Let $\theta$ be an anti-symmetric $\dim\times \dim$ real matrix. 
The noncommutative $\dim$-torus $\A [\dim]$ is the universal $C^\ast$-algebra generated by $\dim$ unitary elements $U_1,\cdots,U_\dim$ which satisfy the commutation relations
\begin{equation*}
U_kU_j=e^{2\pi i \theta_{jk}}U_jU_k.
\end{equation*} 
When $\theta$ is the zero matrix, the $C^\ast$-algebra $\A [\dim]$ is isomorphic to the $C^\ast$-algebra $C(\mathbb{T}^\dim)$ of continuous functions on the $\dim$-dimensional torus $\mathbb{T}^\dim$. 

The algebra of smooth elements  of $\A$, denoted by $\Ai[\dim]$, is a dense involutive subalgebra consisting of elements  of the form 
$$\sum_{\vec{n}=(n_1,\cdots, n_\dim)\in \mathbb{Z}^\dim}a_{\vec{n}}U_1^{n_1}\cdots U_\dim^{n_\dim},$$ 
where $\{a_{\vec{n}}\}_{\vec{n}\in\mathbb{Z}^\dim}$ is a rapidly decaying sequence of complex numbers.
There are $\dim$ closed densely defined derivations on $\A [\dim]$ with the common domain $\Ai[\dim]$ whose values on $U_k$ are given by
\begin{equation*}
\delta_j(U_k)=\delta_{jk}U_k.
\end{equation*} 
We note that $\delta_j(a)^*=-\delta_j(a^*)$, for any $a\in \Ai[\dim]$.
The noncommutative $\dim$-tori possess a  tracial state $\varphi$ which acts on $\Ai[\dim]$ as
\begin{equation*}
\varphi(\sum a_{\vec{n}}U_1^{n_1}\cdots U_\dim^{n_\dim})=a_{\vec{0}}.
\end{equation*}
The trace $\varphi$ is compatible with the derivations $\delta_j$, i.e.  $\varphi\circ \delta_j=0$.
This implies that  a version of the integration by parts is valid on $\nctorus[\dim]$,  
$$\varphi(a\delta_j(b))=-\varphi(\delta_j(a)b),\qquad  a,b\in \Ai[\dim],\, j=1,\cdots,\dim.$$

The subalgebra $\Ai[\dim]$ for $\theta=0$, is indeed the algebra of smooth functions on the $\dim$-dimensional flat torus $\mathbb{T}^d=\mathbb{R}^\dim/(2\pi \mathbb{Z})^\dim$.
In this case, the derivations $\delta_j$ are nothing but  $\frac{1}{i}\frac{\partial}{\partial x_j}$ while the trace $\varphi$ is the (normalized) integration with respect to the  volume form $dx^1\cdots dx^\dim$.

Connes' pseudodifferential calculus on the noncommutative $\dim$-tori can be constructed, as a special case of the general theory developed in \cite{Connes1980} for $C^\ast$-dynamical systems, using the  action of $\mathbb{R}^\dim$ on $\A[\dim]$ by *-automorphisms given by 
\begin{equation*}
\alpha_{\vec{s}}(U_1^{n_1}\cdots U_\dim^{n_\dim})=e^{i\vec{s}\cdot \vec{n}}U_1^{n_1}\cdots U_\dim^{n_\dim},\quad \vec{s}=(s_1,\cdots,s_\dim)\in \mathbb{R}^\dim.
\end{equation*}
We fix the following notations  for a non-negative integer multi-index $\alpha= (\alpha_1,\cdots, \alpha_\dim)$: 
\begin{equation*}
\partial^{\vec{\alpha}}=\frac{\partial^{\alpha_1}}{\partial \xi_1^{\alpha_1}} \cdots \frac{\partial^{\alpha_\dim}}{\partial \xi_\dim^{\alpha_\dim}},
\qquad \delta^{\vec{\alpha}}= \delta_1^{\alpha_1}\cdots \delta_\dim^{\alpha_\dim},
\qquad |\vec{\alpha}|=\sum \alpha_j,\quad \text{and} \quad 
\vec{\alpha}!=\alpha_1!\cdots\alpha_\dim !.
\end{equation*}
A smooth map $\rho: \mathbb{R}^\dim \to \Ai[\dim]$ is called
 a {\it symbol of order $m$}, if 
 \begin{itemize}
 \item[(i)] For every  multi-indices of non-negative integers $\vec{\alpha}$ and $\vec{\beta}$, there exists a constant $C$ such that
\begin{equation*}
 \|\delta^{\vec{\alpha}}  \partial^{\vec{\beta}} \big( \rho(\xi)\big) \|
\leq C (1+|\xi|)^{m-|\vec{\beta}|}.
\end{equation*} 
\item [(ii)]
There exists a smooth map $K: \mathbb{R}^\dim \to \Ai[\dim]$ such that
$$\lim_{\lambda \to \infty} \lambda^{-m} \rho(\lambda\xi_1,\cdots, \lambda\xi_\dim) = K (\xi_1,\cdots, \xi_\dim).$$
\end{itemize}
The space of all symbols of order $m$ is denoted by $S^m(\nctorus[\dim])$.

To every symbol $\rho\in S^m(\nctorus[\dim])$, a pseudodifferential operator $P_\rho$ is assigned which on the elements of $\Ai[\dim]$ is defined as     
\begin{equation*}
P_\rho(a)=\frac1{(2\pi)^{\dim}}\int_{\mathbb{R}^\dim}e^{is\xi}\rho (\xi)\alpha_{\xi}(a)d\xi.
\end{equation*}
The composition $P_{\rho_1}\circ P_{\rho_2}$ of two pseudodifferential operators $P_{\rho_1}$ and $P_{\rho_2}$ with symbols $\rho_j\in S^{m_j}(\nctorus[\dim])$,  is  again a pseudodifferential operator with the product symbol $\rho\in S^{m_1+m_2}(\nctorus[\dim])$  asymptotically given by \cite[Proposition 8]{Connes1980}:
\begin{equation}\label{produtsymbol}
\rho\sim \sum_{\vec{\alpha}\geq 0} 
\partial^{\vec{\alpha}} (\rho_1) \delta^{\vec{\alpha}}(\rho_2)/\vec{\alpha}!.
\end{equation}

For more details on Connes' pseudodifferential calculus we refer the reader to \cite{Connes1980}, for generalized cases to \cite{Lesch-Moscovici2016,Connes-Moscovici2014} and for  recent account on the subject to \cite{Ha-Lee-Ponge2018-I,Ha-Lee-Ponge2018-II}.

\subsection{Rearrangement Lemma}
The rearrangement lemma  appeared in \cite{Connes-Tretkoff2011} as a tool to evaluate integrals of the form
\begin{equation*}
\int_0^\infty (uk^2+1)^{-m}b(uk^2+1)u^{m} du,
\end{equation*}
where $k=e^{h/2}$ for some smooth selfadjoint element  $h \in C^\infty(\nctorus[2])$ and  an element $b$ which do not necessarily commute with $h$.
A version of the lemma, dealing with more than one $b$, then appeared in \cite{Connes-Moscovici2014, Fathizadeh-Khalkhali2012}.
Later, a detailed study of this type of integrals has been carried out in \cite{Lesch2017}, where a new proof for more general version of the lemma is offered and a different point of view  of  such computations is provided.
In \cite{Lesch2017} for $\prod f_j(u,x)$ satisfying an integrability condition (see  \cite[Theorem 3.4]{Lesch2017}) a formula for the integral
\begin{equation*}
\int_0^\infty f_0(u,k)b_1 f_1(u,k)\rho_2\cdots b_n f_n(u,k) du,
\end{equation*}
is given.
This approach uses the multiplication map $\mu :a_1\otimes a_2\otimes \cdots\otimes a_n\mapsto a_1a_2\cdots a_n$  from the projective tensor product $A^{\otimes_\gamma (n)}$ to $A$. 
The above integrals  are expressed as  the contraction of the product of an element $F(k_{(0)},\cdots,k_{(n)})$ of $A^{\otimes_\gamma (n+1)}$, which is given by a multivariate functional calculus, and  the element $b_1\otimes b_2\otimes \cdots b_n\otimes 1$ which is
\begin{equation*}
\mu\Big(F(k_{(0)},\cdots,k_{(n)})(b_1\otimes b_2\otimes \cdots b_n\otimes 1)\Big).
\end{equation*} 
We shall use the convention that such an element will be written as 
\begin{equation}\label{contractionform}
F(k_{(0)},\cdots,k_{(n)})(b_1\cdot b_2\cdot \cdots b_n).
\end{equation}
An element in the form of  \eqref{contractionform} will be said is written in the {\it contraction form}. 
We call the first part $F(h_{(0)},\cdots,h_{(n+1)})$ the {\it commutative part} or {\it operator part} of the element, while the second part, $b_1\cdot b_2 \dots b_n$, will be called the {\it noncommutative part} of the element. 

The following lemma is a version of the rearrangement lemma that we will use in this work and it can be proved exactly  as in  \cite[Theorem 3.4]{Lesch2017} with a slight change in the domain of integration from $[0,\infty)$ to any domain in $\mathbb{R}^N$.

\begin{lemma}[Rearrangement lemma]\label{rearrangmentlemma}
Let $A$ be a unital $C^\ast$-algebra, $h\in A$ be a selfadjoint element, and $\Lambda$ be an open neighborhood of the spectrum of $h$ in $\mathbb{R}$.
For  a domain $U$ in $\mathbb{R}^N$, let 
$f_j:U \times \Lambda\to \mathbb{C}$ for $0\leq j\leq n$  be  smooth functions  such that $f(u,\lambda)=\prod_{j=0}^n f_j(u,\lambda_j)$ satisfies the following integrability condition:
for any compact subset $K\subset \Lambda^{n+1}$ and every given  integer multi-index $\alpha$ we have    
$$\int_U \sup\limits_{\lambda\in K}|\partial_\lambda^\alpha f(u,\lambda)|du<\infty.$$
Then,
\begin{equation}\label{formulaofrearrangement}
\int_U f_0(u,h)b_1f_1(u,h)\cdots b_n f_n(u,h)du=F(h_{(0)},h_{(1)},\cdots,h_{(n)})(b_1\cdot b_2\cdots b_n),
\end{equation}
where $F(\lambda)=\int_U f(u,\lambda)du$.\qed
\end{lemma}
\noindent For more information on the functional calculus $F(h_{(0)},h_{(1)},\cdots,h_{(n)})$ on $A^{\otimes_{\gamma} (n+1)}$ and its relation to  the Bochner integral $\int_U f(u,h_{(0)},h_{(1)},\cdots,h_{(n)})du$,  we refer the reader to \cite[Section 3.2]{Lesch2017}.  
\begin{remark}\label{integralform}
Equation \eqref{formulaofrearrangement} shows that integration is a source to generate elements in the contraction form.
Using lemma \ref{rearrangmentlemma}, as noted in \cite[Remark 3.3]{Lesch2017},  every expression in the contraction form with a Schwartz function $F\in\mathcal{S}(\mathbb{R}^{n+1})$ used in the operator part can indeed be obtained using an integral such as the one in the left hand side of equation \eqref{formulaofrearrangement}.
To see this, it is enough to set 
$$f_n(\xi,\lambda)=\hat{f}(\xi)e^{i\xi_n \lambda},\quad f_j(\xi,\lambda)=e^{i\xi_j \lambda},\quad 0\leq j\leq n-1,$$ 
and $f(\xi,\lambda_0,\cdots,\lambda_n)=\prod_{j=0}^n f_j(\xi,\lambda_j)$. 
By the Fourier inversion formula, we have 
$F(\lambda)=\int f(\xi,\lambda)d\xi$.
Then,  Lemma \ref{rearrangmentlemma}  gives the equality 
\begin{equation}\label{integralformformula}
F(h_{(0)},\cdots,h_{(n)})(b_1\cdot b_2\cdots  b_n)=\int_{\mathbb{R}^{n}}e^{i\xi_0 h}b_1e^{i\xi_1 h}b_2\cdots b_n e^{i\xi_n h}\hat{f}(\xi)d\xi.
\end{equation} 
This fact will be used several times in this work.
\end{remark}


\subsection{Newton divided differences and closed derivations}\label{closedderandfinite}
As an application of the rearrangement lemma, we find a formula for the derivative of a  smooth element written in contraction form in terms of the (partial) divided differences of the function used in the functional part and the derivatives of the noncommutative part.
We first briefly review the Newton divided differences. 

Let $x_0, x_1,\cdots, x_n$ be distinct points in an interval $I\subset \mathbb{R}$ and let $f$ be a function on $I$. 
The {\it $n^{th}$-order Newton divided difference} of $f$, denoted by $[x_0,x_1,\cdots,x_n;f]$,
is the coefficient of $x^n$ in the interpolating polynomial of $f$ at the given points.
In other words, if the interpolating polynomial is $p(x)$ then 
\begin{equation*}
p(x)=p_{n-1}(x)+[x_0,x_1,\cdots,x_n;f](x-x_0)\cdots(x-x_{n-1}),
\end{equation*}
where $p_{n-1}(x)$ is a polynomial of    degree at most $n-1$.
The following is the recursive formula for the divided difference which justifies the name, divided difference.
\begin{equation*}
\begin{aligned}[]
  [ x_0; f]&=f(x_0)\\
[x_0,x_1,\cdots,x_n;f]&=\frac{[x_1,\cdots,x_n;f]-[x_0,x_1,\cdots,x_{n-1};f]}{x_n-x_0}.
\end{aligned}
\end{equation*}
The divided difference is also given by the explicit formula
\begin{equation*}
[x_0,x_1,\cdots,x_n;f]=\sum_{j=0}^n \frac{f(x_j)}{\prod_{j\neq l} (x_j-x_l)}.
\end{equation*}
It is clear that the $n^{th}$-order divided difference, as a function of $x_j$'s, is a symmetric function.
There are different integral formulas for the divided differences from which we will use the {\it Hermite-Genocchi formula} (cf. e.g. \cite[Theorem 3.3]{Atkinson1989}). 
For an $n$ times continuously differentiable function $f$, Hermite-Genocchi formula gives
\begin{equation}\label{HermiteGenocchi}
[x_0,\cdots,x_n;f]=\int_{\Sigma_n} f^{(n)}\Big(\sum_{j=0}^n s_j x_j\Big)ds,
\end{equation}
where $\Sigma_n$ denotes the standard $n$-simplex  
$$\Sigma_n=\Big\{(s_{0},\dots ,s_{n})\in\mathbb{R}^{n+1}{\Big |}~\sum _{j=0}^{n}s_{j}=1{\mbox{ and }}s_{j}\geq 0{\mbox{ for all }}j\Big\}.$$
We refer the reader to \cite[Chapter 1]{Milne1951} or \cite[Section 3.2]{Atkinson1989} for more details on the Newton divided differences.

Let $\delta$ be a (densely defined, unbounded) closed derivation on a $C^\ast$-algebra $A$.
If $a\in \dom(\delta)$, then $e^{za}\in \dom(\delta)$ for any $z\in \mathbb{C}$ \cite[p. 68]{Sakai1991}  and
\begin{equation}\label{generalduahmmel}
\delta(e^{za})=z\int_0^1 e^{zsa}\delta(a) e^{z(1-s)a}ds.
\end{equation}
Using the rearrangement lemma, one can go one step further and show that the derivation of $f(h)$ can be presented in the contraction form.  
In \cite{Connes-Tretkoff2011} this result is used for $f(x)=\exp(x)$ as a corollary of the expansional formula for $e^{A+B}$ and
in \cite{Lesch2017}, the expansional formula is proved for general functions \cite[Proposition 3.7]{Lesch2017}.
The following theorem gives the result not only for $f(h)$ but for any element  written in the contraction form.

\begin{theorem}\label{derivationoffunctionofh}
Let  $\delta$ be a closed derivation of a $C^\ast$-algebra $A$ and $h\in \dom(\delta)$ be a selfadjoint element. 
\begin{itemize}
\item[$i)$] Let $f:\mathbb{R}\to \mathbb{C}$ be a smooth function. Then, $f(h)\in \dom(\delta)$ and 
\begin{equation*}
\delta(f(h))=[h_{(0)},h_{(1)};f](\delta(h)).
\end{equation*}
\item[$ii)$] Let $b_j\in \dom(\delta)$,   $1\leq j \leq n$, and let $f:\mathbb{R}^{n+1}\to \mathbb{C}$ be a smooth function.
Then $f(h_{(0)},\cdots, h_{(n)})(b_1\cdot b_2\cdot \dots\cdot b_n)$ is in the domain of $\delta$  and  
\begin{equation*}
\begin{aligned}
&\delta(f(h_{(0)},\cdots,h_{(n)})(b_1\cdot b_2 \cdots  b_n))\\
&=\sum_{j=1}^{n} f(h_{(0)},\cdots,h_{(n)})\big(b_1\cdots b_{j-1}\cdot\delta(b_j)\cdot b_{j+1} \cdots  b_n\big)\\
&\quad +\sum_{j=0}^n f_j(h_{(0)},\cdots,h_{(n+1)})\big(b_1\cdots b_{j}\cdot\delta(h)\cdot b_{j+1} \cdots  b_n\big),
\end{aligned}
\end{equation*}
where $f_j(t_0,\cdots,t_{n+1})$, which we call the partial divided difference, is defined as 
$$f_j(t_0,\cdots,t_{n+1})=\left[t_j,t_{j+1};t\mapsto f(t_0,\cdots, t_{j-1},t,t_{j+2},\cdots,t_n)\right].$$
\end{itemize}
\end{theorem}
\begin{proof}
It is enough to assume that $f$ is a Schwartz class function on $\mathbb{R}$. We have
\begin{equation*}
f(h)=\int_\mathbb{R} e^{ih\xi}\hat{f}(\xi)d\xi,
\end{equation*}
where $\hat{f}$ is the Fourier transform of the function $f$.
By closedness of $\delta$, we can then change the order of integration and derivation: 
\begin{eqnarray*}
\delta(f(h))
&=&\int_\mathbb{R} \delta(e^{ih\xi})\hat{f}(\xi)d\xi\\
&=&\int_\mathbb{R} i\xi\int_0^1 e^{i\xi s h}\delta(h) e^{i\xi(1-s)h}ds\hat{f}(\xi)d\xi\\
&=&\mathcal{C}\left(\int_\mathbb{R} \int_0^1 e^{i\xi s h}\otimes e^{i\xi(1-s)h}i\xi\hat{f}(\xi)dsd\xi(\delta(h)\otimes 1)\right)\\
&=&F(h_{(0)},h_{(1)})(\delta(h)).
\end{eqnarray*}
Here $F(t_0,t_1)$ is the function given by the integral 
$$\int_\mathbb{R} \int_0^1 e^{i\xi s t_0}e^{i\xi(1-s)t_1}i\xi\hat{f}(\xi)dsd\xi.$$ 
Using Hermite-Genocchi formula \eqref{HermiteGenocchi} and the properties of the Fourier transform, we have  
\begin{equation*}
F(t_0,t_1)=\int_0^1 \int_\mathbb{R} e^{i\xi (s t_0+\xi(1-s)t_1)}i\xi\hat{f}(\xi)d\xi ds=\int_0^1 f'(s t_0+\xi(1-s)t_1)ds=[t_0,t_1;f].
\end{equation*}

We prove the part $ii)$ only for $n=1$; the proof of the general case follows similarly.   
We first use Remark \ref{integralform} to write $f(h_{(0)},h_{(1)})(b_1)$ as an integral 
$$f(h_{(0)},h_{(1)})(b_1)=\int_{\mathbb{R}^2} e^{i\xi_1h} b_1 e^{i\xi_2h}\hat{f}(\xi)d\xi.$$
The Leibniz  rule together with the fact that $\delta$ is closed  imply that 
\begin{align*}
&\delta\Big(f(h_{(0)},h_{(1)})(b_1)\Big)
= \delta\Big(\int_{\mathbb{R}^2} e^{i\xi_1h}b_1 e^{i\xi_2h}\hat{f}(\xi)d\xi\Big)\\
&= \int_{\mathbb{R}^2} \delta\Big(e^{i\xi_1h}\Big)b_1 e^{i\xi_2h}\hat{f}(\xi)d\xi
+ \int_{\mathbb{R}^2} e^{i\xi_1h}\delta\Big(b_1\Big) e^{i\xi_2h}\hat{f}(\xi)d\xi
+ \int_{\mathbb{R}^2} e^{i\xi_1h}b_1 \delta\Big(e^{i\xi_2h}\Big)\hat{f}(\xi)d\xi\\
&= \int_{\mathbb{R}^2}\int_0^1 i\xi_1 e^{is\xi_1h}\delta(h)e^{i(1-s)\xi_1h}b_1 e^{i\xi_2h}\hat{f}(\xi)dsd\xi\\
&\quad + \int_{\mathbb{R}^2} e^{i\xi_1h}\delta(b_1) e^{i\xi_2h}\hat{f}(\xi)d\xi\\
&\quad + \int_{\mathbb{R}^2}\int_0^1 i\xi_2 e^{i\xi_1h}b_1 e^{is\xi_2h}\delta(h)e^{i(1-s)\xi_2h}\hat{f}(\xi)dsd\xi\\
&=f_1(h_{(0)},h_{(1)},h_{(2)})(b_1\cdot\delta(h))+f(h_{(0)},h_{(1)})(\delta(b_1))
+f_2(h_{(0)},h_{(1)},h_{(2)})( \delta(h)\cdot b_1),
\end{align*}
where 
\begin{equation*}
\begin{aligned}
f_1(t_0,t_1,t_2)&=\int_{\mathbb{R}^2}\int_0^1 i\xi_1 e^{is\xi_1t_0}e^{i(1-s)\xi_1t_1} e^{i\xi_2t_2}\hat{f}(\xi)dsd\xi=[t_0,t_1;f(\cdot,t_2)],\\ 
f_2(t_0,t_1,t_2)&=\int_{\mathbb{R}^2}\int_0^1 i\xi_2 e^{i\xi_1t_0}e^{is\xi_2t_1}e^{i(1-s)\xi_2t_0}\hat{f}(\xi)dsd\xi=[t_1,t_2;f(t_0,\cdot)].
\end{aligned}\vspace*{-0.85cm}
\end{equation*}
\end{proof}

\begin{corollary}\label{secondderivationoffunctionofh}
Let $h$ be a selfadjoint element in a $C^\ast$-algebra  $A$ and let $\delta_1$ and $\delta_2$ be two closed derivations.
If $h\in \dom(\delta_2)$ and  $\delta_2(h)\in \dom(\delta_1)$, then for any smooth function $f$ we have
\begin{equation*}
\delta_1\delta_2(f(h))=\left[h_{(0)},h_{(1)},h_{(2)};f\right]\left(\delta_2(h)\cdot \delta_1(h)+\delta_1(h)\cdot \delta_2(h)\right)
+\left[h_{(0)},h_{(1)};f\right](\delta_1\delta_2(h)).
\end{equation*}
\end{corollary}
\begin{proof}
Using Theorem \ref{derivationoffunctionofh}, we have 
\begin{equation*}
\begin{aligned}
\delta_1\delta_2(f(h))=&\delta_1\left(\delta_2(f(h))\right)
=\delta_1\left(\left[h_{(0)},h_{(1)};f\right](\delta_2(h))\right)\\
=&\left[h_{(0)},h_{(1)};f\right](\delta_1\delta_2(h))\\
&+\left[h_{(0)},h_{(1)};t\mapsto [t,h_{(2)};f]\right]\left(\delta_1(h)\cdot \delta_2(h)\right)\\
&+\left[h_{(1)},h_{(2)};t\mapsto [h_{(0)},t;f]\right]\left(\delta_2(h)\cdot \delta_1(h)\right).
\end{aligned}
\end{equation*}
The functional part of the last two terms, by the definition of second order divided difference and invariance of divided difference under the permutation of variables $t_j$'s,  are equal to  $\left[h_{(0)},h_{(1)},h_{(2)};f\right]$.
The proof is done using the fact that the contraction is linear with respect to the noncommutative part.
\end{proof}
This corollary can also be deduced from the variational formula used in \cite[Example 3.9]{Lesch2017}.
In the case of the noncommutative two tori, the conditions of Corollary \ref{secondderivationoffunctionofh} are automatically  satisfied for any smooth selfadjoint element $h=h^*\in\Ai$. 
For $f(t)=e^{\frac{t}2}$, we have 
\begin{equation*}
\begin{aligned}
\left[t_0,t_1,t_2;e^{\frac{t}2}\right]&=\frac{e^{t_0/2}}{(t_0-t_1)(t_0-t_2)}+\frac{e^{t_1/2}}{(t_1-t_0)(t_1-t_2)}+\frac{e^{t_2/2}}{(t_2-t_0)(t_2-t_1)}\\
&=e^{t_0/2}\frac{(t_1-t_2)+(t_2-t_0)e^{(t_1-t_0)/2}+(t_0-t_1)e^{(t_2-t_0)/2}}{(t_0-t_1)(t_0-t_2)(t_1-t_2)},\\
\left[t_0,t_1;e^{\frac{t}2}\right]&=\frac{e^{t_0/2}-e^{t_1/2}}{t_0-t_1}
=e^{t_0/2}\frac{1-e^{(t_1-t_0)/2}}{t_0-t_1}.
\end{aligned}
\end{equation*} 
This is in complete agreement with the formula (6.8) in \cite{Connes-Moscovici2014} by setting 
$$s=t_1-t_0, \quad  t=t_2-t_1, \text{ and }  h=2\log(k).$$

\section{Heat trace asymptotics for  Laplace type \texorpdfstring{$h$}{h}-differential operators}
Following the strategy underlined in the previous section, we shall compute the first three terms of the symbol of the parametrix of an elliptic operator on the noncommutative tori with a positive definite principal symbol whose symbol is given in the contraction form. 
These terms will be given in the contraction form as well.    

\subsection{Laplace type \texorpdfstring{$h$}{h}-operators and their paramatrix}\label{rawcomp}
In the following definition we introduce a new class of differential operators on noncommutative tori that generalize the already studied classes.
\begin{definition}
Let $h\in\Ai[\dim]$ be a smooth selfadjoint element.
\begin{itemize}
\item[$i)$] By an {\it $h$-differential operator} on $\nctorus[\dim]$, we mean a differential operator 
$P=\sum_{\vec{\alpha}}p_{\vec{\alpha}}\delta^{\vec{\alpha}},$ 
with $\Ai[\dim]$-valued coefficients $p_{\vec{\alpha}}$ which can be written in the contraction form 
$$p_{\vec{\alpha}}=P_{\vec{\alpha},\vec{\alpha}_1,\cdots,\vec{\alpha}_k}(h_{(0)},\cdots,h_{(k)})(\delta^{\vec{\alpha}_1}(h)\cdots \delta^{\vec{\alpha}_k}(h)).$$
\item[$ii)$] We call a second order $h$-differential operator  $P$ a {\it Laplace type $h$-differential operator} if its symbol is a sum of homogeneous parts $p_j$ of the form
\begin{equation}\label{noncommutativelaplacetypesym}
\begin{aligned}
p_2&=P_2^{ij}(h)\xi_i\xi_j,\\
p_1&=  P_1^{ij}(h_{(0)},h_{(1)})(\delta_i(h)) \xi_j,\\
p_0&= P_{0,1}^{ij}(h_{(0)},h_{(1)})(\delta_i\delta_j(h))+P_{0,2}^{ij}(h_{(0)},h_{(1)},h_{(2)})(\delta_i(h)\cdot \delta_j(h)),
\end{aligned}
\end{equation}
where the principal symbol $p_2:\mathbb{R}^\dim\to \Ai[\dim]$ is a $\Ai[\dim]$-valued quadratic form such that $p_2(\xi)>0$ for all $\xi\in\mathbb{R}^\dim$. 
\end{itemize} 
\noindent We can allow  the symbols to be matrix valued, that is  $p_j:\mathbb{R}^\dim\to \Ai[\dim]\otimes M_n(\mathbb{C})$, provided that all $p_2(\xi)\in \Ai[\dim]\otimes {\rm I}_n$ for all nonzero $\xi\in\mathbb{R}^\dim$.
\end{definition}

\begin{example}
Many of the differential operators on noncommutative tori which  were studied in the literature 
 are Laplace type $h$-differential operators. 
 For instance, the two differential operators on $\nctorus[2]$ whose spectral invariants are studied in \cite{Fathizadeh-Khalkhali2013,Connes-Moscovici2014} are indeed Laplace type $h$-differential operator.
Using the notation of \cite[Lemma 1.11]{Connes-Moscovici2014}, $k=e^{h/2}$, 
these operators are given by
\begin{equation*}
k\lap k=k\delta\delta^* k,\qquad \lap_\varphi^{(0,1)}=\delta^* k^2 \delta,
\end{equation*}
where, $\delta=\delta_1+\bar{\tau}\delta_2$ and  $\delta^*=\delta_1+\tau\delta_2$ for some complex number $\tau$ in the upper half plane.
The operator $k\lap k$ is antiunitary equivalent to the Laplacian on the functions on $\nctorus[2]$ equipped with the conformally flat metric $k^2\begin{pmatrix}
1 & \Re(\tau)\\ \Re(\tau) & |\tau|^2
\end{pmatrix}$.
We shall construct this operator in all dimensions for the more general class of metrics, which will be introduced in Section \ref{functionalmetricsec}, and study it in details then. 
On the other hand, the symbol of  the other operator $\lap_\varphi^{(0,1)}$ is the sum of the following homogeneous terms  \cite[Lemma 6.1]{Connes-Moscovici2014} :
\begin{align*}
p_2(\xi)=& k^2(\xi_1^2+2\Re(\tau)\xi_1\xi_2+|\tau|^2\xi_2^2),\\
p_1(\xi)=&
(\delta_1(k^2)+\tau \delta_2(k^2))\xi_1+(\bar{\tau}\delta_1(k^2)+|\tau|^2 \delta_2(k^2))\xi_2,\\
p_0(\xi)=&0.
\end{align*}
Clearly, we have  
\begin{align*}
p_2^{11}(t)= e^t,
\qquad p_2^{12}(t)=p_2^{21}(t)= \Re(\tau)e^t,
\qquad p_2^{22}(t)=|\tau|^2 e^t,
\end{align*}
and using Corollary \ref{secondderivationoffunctionofh}, we obtain the other functions 
\begin{align*}
P^{1,1}_1(t_0,t_1)=&\left[t_0,t_1;e^{u}\right],\qquad && P^{2,1}_1(t_0,t_1)=\tau \left[t_0,t_1;e^{u}\right],\\
P^{1,2}_1(t_0,t_1)=&\bar{\tau}\left[t_0,t_1;e^{u}\right],\qquad && P^{2,2}_1(t_0,t_1)=|\tau|^2 \left[t_0,t_1;e^{u}\right].
\end{align*}
\end{example}
\begin{remark} 
The Laplace type $h$-operators can also be made sense on smooth vector bundles on manifolds, where $h$ can be taken (locally) as a function or an endomorphism of the bundle. 
If $h$ is a function, then a Laplace type $h$-differential operator is indeed a Laplace type operator.
However, if $h$ is an endomorphism, these operators belong to a different class of elliptic operators of order two for which some aspects of its  heat trace densities are studied in   \cite{Iochum-Masson2018,Avramidi-Branson2001}.
The  matrix-valued conformally perturbed Dirac operator on smooth manifolds, which is an $h$-differential operator, is also studied recently in \cite{Khalkhali-Sitarz2018}.
\end{remark}

Let  $P$ be a positive Laplace type $h$-differential operator. 
For any $\lambda$ in $\mathbb{C}\backslash \mathbb{R}^{\geq 0}$,  the operator $(P-\lambda)^{-1}$ is a classical pseudodifferential operator of order $-2$, i.e. its symbol can be written  as 
$$\sigma\left((P-\lambda)^{-1}\right)=b_0(\xi,\lambda)+b_1(\xi,\lambda)+b_2(\xi,\lambda)+\cdots,$$ 
where $b_j$ is a homogeneous  symbol of order $-2-j$.
The  following recursive formula for $b_j$'s can be obtained from the product formula for symbols  \eqref{produtsymbol}.
\begin{equation}\label{recursiveformulaforpara}
\begin{aligned}
b_0(\xi,\lambda)&=(p_2(\xi)-\lambda)^{-1},\\
b_n(\xi,\lambda)&=-\sum_{j=0}^{n-1}\sum_{k=0}^2\sum_{|\vec{\alpha}|=n-j+k-2} \partial^{\vec{\alpha}}b_j(\xi,\lambda) \, \delta^{\vec{\alpha}}(p_k(\xi))\, b_0(\xi,\lambda)/\vec{\alpha}!  \qquad n>0.
\end{aligned}
\end{equation}
If we start off with symbols written in the contraction form, the formula \eqref{recursiveformulaforpara} shows   that all $b_n(\xi,\lambda)$'s can also be computed in the contraction form using Theorem \ref{derivationoffunctionofh}.
In the rest of this section, we compute the terms $b_0(\xi,\lambda)$, $b_1(\xi,\lambda)$ and $b_2(\xi,\lambda)$ in the contraction form as follows. 

\noindent{\bf Term $b_0(\xi,\lambda)$}: 
This term is defined as $b_0(\xi,\lambda)=(p_2(\xi)-\lambda)^{-1}$ and it can be written as 
\begin{equation*}
b_0(\xi,\lambda)=B_0(h),\qquad B_0(t_0)=(P_2^{ij}(t_0)\xi_i\xi_j-\lambda)^{-1}.
\end{equation*}
We choose to overlook the dependence of $B_0$ on the variables $\xi$ and $\lambda$ in its notation.

\noindent {\bf Term $b_1(\xi,\lambda)$}:
Using \eqref{recursiveformulaforpara}, we have 
\begin{equation*}
b_1=-b_0(\xi,\lambda)p_1(\xi)b_0(\xi,\lambda)-\sum \partial_i b_0(\xi,\lambda)\, \delta_i(p_2(\xi))b_0(\xi,\lambda).
\end{equation*}
We then write each term of the above expression in the contraction form: 
\begin{align*}
-b_0(\xi,\lambda)p_1(\xi)b_0(\xi,\lambda)&=-\xi_jB_0(h_{(0)})P_1^{ij}(h_{(0)},h_{(1)})B_0(h_{(1)})(\delta_i(h)),\\
\partial_i b_0(\xi,\lambda) &=-2P_2^{ij}(h)B_0^2(h)\xi_j,\\
\delta_i(p_2(\xi))&=\xi_k\xi_\ell \, \big[h_{(0)},h_{(1)};P_2^{k\ell}\big](\delta_i(h)),\\
-\sum \partial_i b_0(\xi,\lambda)\, \delta_i(p_2(\xi))b_0(\xi,\lambda)&=2\xi_j\xi_k\xi_\ell P_2^{ij}(h_{(0)})B_0^2(h_{(0)}) \big[h_{(0)},h_{(1)};P_2^{k\ell}\big] B_0(h_{(1)})(\delta_i(h)).
\end{align*}
Therefore, $b_1(\xi,\lambda)=B_1^{i}(\xi,\lambda,h_{(0)},h_{(1)})(\delta_i(h))$, where
\begin{equation*}
B_1^i(\xi,\lambda,t_0,t_1)= -\xi_jB_0(t_0)P_1^{ij}(t_0,t_1)B_0(t_1)+2\xi_j\xi_k\xi_\ell P_2^{ij}(t_0)B_0^2(t_0) \big[t_0,t_1;P_2^{k\ell}\big] B_0(t_1).
\end{equation*}

\noindent{\bf Term $b_2(\xi,\lambda)$}:
Formula \eqref{recursiveformulaforpara} for $n=2$, gives $b_2(\xi,\lambda)$ in terms of $b_0(\xi,\lambda)$ and $b_1(\xi,\lambda)$ and $p_j(\xi)$'s:  
\begin{equation}\label{formulab2firstversion}
\begin{aligned}
b_2(\xi,\lambda)
=&-b_0(\xi,\lambda)p_0(\xi)b_0(\xi,\lambda)-b_1(\xi,\lambda)p_1(\xi)b_0(\xi,\lambda)\\
&
-\sum \partial_ib_0(\xi,\lambda)\, \delta_i(p_1(\xi))b_0(\xi,\lambda)
- \sum \partial_i b_1(\xi,\lambda)\, \delta_i(p_2(\xi))b_0(\xi,\lambda)\\
&-\sum \frac12 \partial_i\partial_j b_0(\xi,\lambda) \delta_i\delta_j(p_2(\xi))b_0(\xi,\lambda).
\end{aligned}
\end{equation}
To compute $b_2(\xi,\lambda)$ we  need $\partial_j b_1(\xi,\lambda)$. 
First, we should note that  
\begin{equation*}
\partial_i b_1(\xi,\lambda)=\partial_j\left(B_1^{i}(\xi,\lambda,h_{(0)},h_{(1)})(\delta_i(h))\right)=\partial_j\left(B_1^{i}(\xi,\lambda,h_{(0)},h_{(1)})\right)(\delta_i(h)).
\end{equation*}
Hence, we have
\begin{equation}\label{derivativeofb1}
\begin{aligned}
\partial_jB_1^{i}(\xi,\lambda,t_0,t_1)
=& -B_0(t_0)P_1^{ij}(t_0,t_1)B_0(t_1) \\
& +2 \xi_k \xi_\ell P_2^{j\ell}B_0^2(t_0) \, P_1^{ik}(t_0,t_1) \, B_0(t_1) \\
&+2 \xi_k \xi_\ell B_0(t_0)P_1^{ik}(t_0,t_1)\,  P_2^{j\ell}B_0^2(t_1) \\
&+ 2\, \xi_\ell\xi_m P_2^{ij}B_0^2(t_0) \big[t_0,t_1;P_2^{\ell m}\big] B_0(t_1)\\ 
&+ 4\xi_k\xi_\ell P_2^{ik}(t_0)B_0^2(t_0) \big[t_0,t_1;P_2^{ j\ell}\big] B_0(t_1)\\
&- 8\xi_k\xi_\ell\xi_m\xi_n P_2^{ik}(t_0) P_2^{jn}(t_0) B_0^3(t_0) \, \big[t_0,t_1;P_2^{\ell m}\big] B_0(t_1)\\
&- 4\xi_k\xi_\ell\xi_m\xi_n P_2^{ik}(t_0)B_0^2(t_0) \big[t_0,t_1;P_2^{\ell m}\big]\,  P_2^{jn}(t_1) B_0^2(t_1).
\end{aligned}
\end{equation}

We now compute each term of expression \eqref{formulab2firstversion} separately:
\begin{align*}
-b_0(\xi,\lambda)p_0(\xi)b_0(\xi,\lambda)
=& -B_0(h_{(0)})P_{0,1}^{ij}(h_{(0)},h_{(1)})B_0(h_{(1)})(\delta_i\delta_j(h))\\
&-B_0(h_{(0)})P_{0,2}^{ij}(h_{(0)},h_{(1)},h_{(2)})B_0(h_{(2)})(\delta_i(h)\cdot\delta_j(h)).
\end{align*}
The second term of the sum \eqref{formulab2firstversion} is equal to
\begin{align*}
&-b_1(\xi,\lambda)p_1(\xi)b_0(\xi,\lambda)=\\
&\Big(\xi_k\xi_\ell B_0(h_{(0)})P_1^{ik}(h_{(0)},h_{(1)})B_0(h_{(1)})P^{j\ell}_1(h_{(1)},h_{(2)})B_0(h_{(2)})\\
& \,-2\xi_k\xi_\ell\xi_m\xi_n P_2^{im}B_0^2(h_{(0)}) \big[h_{(0)},h_{(1)};P_2^{k\ell}]\, B_0(h_{(1)})P^{jn}_1(h_{(1)},h_{(2)})B_0(h_{(2)})\Big)(\delta_i(h)\cdot \delta_j(h)).
\end{align*}
We need $\delta_i(p_1(\xi))$ for the next term;
\begin{align*}
\delta_i(p_1(\xi))
=&\xi_\ell P_1^{j\ell}(h_{(0)},h_{(1)})\,\big(\delta_i\delta_j(h)\big)\\
&+\xi_\ell \big[h_{(0)},h_{(1)};P_1^{j\ell}(\cdot,h_{(2)})\big]\,\big(\delta_i(h)\cdot\delta_j(h)\big)\\
&+\xi_\ell \big[h_{(1)},h_{(2)};P_1^{j\ell}(h_{(0)},\cdot)\big]\,\big(\delta_j(h)\cdot\delta_i(h)\big).
\end{align*}
Then, the third term of \eqref{formulab2firstversion} is given by
\begin{align*}
&-\sum\partial_i b_0(\xi,\lambda)\delta_i(p_1(\xi))b_0(\xi,\lambda)\\
&=2\xi_k\xi_\ell  P_2^{ik}B_0^2(h_{(0)})P_1^{j\ell}(h_{(0)},h_{(1)})B_0(h_{(1)})\,\big(\delta_i\delta_j(h)\big)\\
&\,\, +2\xi_k\xi_\ell  P_2^{ik}B_0^2(h_{(0)}) \big[h_{(0)},h_{(1)};P_1^{j\ell}(\cdot,h_{(2)})\big]B_0(h_{(2)})\,\big(\delta_i(h)\cdot\delta_j(h)\big)\\
&\,\,+2\xi_k\xi_\ell  P_2^{ik}B_0^2(h_{(0)})\,  \big[h_{(1)},h_{(2)};P_1^{j\ell}(h_{(0),\cdot})\big]\, B_0(h_{(2)})\,\big(\delta_j(h)\cdot\delta_i(h)\big).
\end{align*}
Considering \eqref{derivativeofb1}, it can  readily be seen that $-\sum_j \partial_j(b_1)\delta_j(p_2)b_0$ is of the form 
$$F^{ij}(h_{(0)},h_{(1)},h_{(2)})(\delta_i(h)\cdot \delta_j(h)),$$
 where 
\begin{align*}
F^{ij}(t_0,t_1,t_2)
&= \xi_k \xi_\ell B_0(t_0)P_1^{ij}(t_0,t_1)B_0(t_1)\, \big[t_1,t_2;P_2^{k\ell}\big]\, B_0(t_2) \\
& -2 \xi_k \xi_\ell\xi_m\xi_n P_2^{j\ell}(t_0)\,B_0^2(t_0) \, P_1^{ik}(t_0,t_1) \, B_0(t_1)\big[t_1,t_2;P_2^{mn}\big]\, B_0(t_2)  \\
&-2 \xi_k \xi_\ell\xi_m\xi_n B_0(t_0)P_1^{ik}(t_0,t_1)\,  P_2^{j\ell}(t_1)B_0^2(t_1) \, \big[t_1,t_2;P_2^{mn}\big]\, B_0(t_2)  \\
&- 2\, \xi_k\xi_\ell\xi_m\xi_n P_2^{ij}(t_0)B_0^2(t_0) \,\big[t_0,t_1;P_2^{\ell m}\big]\, B_0(t_1)\, \big[t_1,t_2;P_2^{kn}\big]\, B_0(t_2)  \\ 
&- 4\xi_k\xi_\ell\xi_m\xi_n P_2^{ik}(t_0)B_0^2(t_0) \,\big[t_0,t_1;P_2^{ j\ell}\big]\, B_0(t_1)\, \big[t_1,t_2;P_2^{mn}\big]\, B_0(t_2) \\
&+ 8\xi_k\xi_\ell\xi_m\xi_n\xi_p\xi_q P_2^{ik}(t_0)P_2^{jn}(t_0)B_0^3(t_0) \,\big[t_0,t_1;P_2^{\ell m}\big]\, B_0(t_1)\, \big[t_1,t_2;P_2^{pq}\Big]\, B_0(t_2) \\
&+ 4\xi_k\xi_\ell\xi_m\xi_n\xi_p\xi_q P_2^{ik}(t_0)B_0^2(t_0) \,\big[t_0,t_1;P_2^{\ell m}\big]\,  P_2^{jn}(t_1)\,B_0^2(t_1)\, \big[t_1,t_2;P_2^{pq}\big]B_0(t_2).
\end{align*}
To compute the very last term $-\sum_{i,j} \frac12 \partial_i\partial_j b_0(\xi,\lambda)\delta_i\delta_j(p_2(\xi))b_0(\xi,\lambda)$, we first write  the terms $ \frac12 \partial_i\partial_jb_0(\xi,\lambda)$ and $\delta_i\delta_j(p_2(\xi))$ in the contraction form:
\begin{align*}
 \frac12 \partial_i\partial_j b_0(\xi,\lambda)
 &=
 -P_2^{ij}(h) B_0^2(h)+4\xi_k \xi_\ell P_2^{ik} P_2^{j\ell}(h)B_0^3(h),
\end{align*}
and
\begin{align*}
\delta_i\delta_j(p_2(\xi))
&=\xi_m\xi_n \,\big[h_{(0)},h_{(1)};P_2^{mn}\big]\,\big(\delta_i\delta_j(h)\big)\\
&+\xi_m\xi_n \,\big[h_{(0)},h_{(1)},h_{(2)};P_2^{mn}\big]\,\big(\delta_i(h)\cdot\delta_j(h)\big)\\
&+\xi_m\xi_n \,\big[h_{(0)},h_{(1)},h_{(2)};P_2^{mn}\big]\,\big(\delta_j(h)\cdot\delta_i(h)\big).
\end{align*}
Therefore, the last term of \eqref{formulab2firstversion}  is given by the following formula:
\begin{align*}
&-\sum_{i,j} \frac12 \partial_i\partial_j b_0(\xi,\lambda) \delta_i\delta_j(p_2(\xi))b_0(\xi,\lambda)=\\
&\qquad+\xi_k\xi_\ell P_2^{ij} B_0^2(h_{(0)}) \,\big[h_{(0)},h_{(1)};P_2^{k\ell}\big]\, B_0(h_{(1)})\big(\delta_i\delta_j(h)\big)\\
&\qquad-4\xi_k \xi_\ell\xi_m\xi_n P_2^{ik} P_2^{j\ell}(h_{(0)})B_0^3(h_{(0)}) \,\big[h_{(0)},h_{(1)};P_2^{mn}\big] \, B_0(h_{(1)})\big(\delta_i\delta_j(h)\big)\\ \nonumber
&\qquad+2\xi_k\xi_\ell P_2^{ij}B_0^2(h_{(0)}) \big[h_{(0)},h_{(1)},h_{(2)};P_2^{k\ell}\big] \, B_0(h_{(2)})\big(\delta_i(h)\cdot\delta_j(h)\big)\\ 
&\qquad-8\xi_k \xi_\ell \xi_m\xi_n P_2^{ik} P_2^{j\ell}B_0^3(h_{(0)}) \,\big[h_{(0)},h_{(1)},h_{(2)};P_2^{mn}\big] \, B_0(h_{(2)})\big(\delta_i(h)\cdot\delta_j(h)\big).
\end{align*}
Adding  all the terms, $b_2(\xi,\lambda)$  is found to be
\begin{equation*}
b_2(\xi,\lambda)=B_{2,1}^{ij}(\xi,\lambda,h_{(0)},h_{(1)})\big(\delta_i\delta_j(h)\big)+B_{2,2}^{ij}(\xi,\lambda,h_{(0)},h_{(1)},h_{(2)})\big(\delta_i(h)\cdot\delta_j(h)\big),
\end{equation*}
where the functions $B_{2,k}^{ij}$ are given by   
\begin{align*}
B_{2,1}^{ij}(\xi,\lambda,t_0,t_1)
=& -B_0(t_0)P_{0,1}^{ij}(t_0,t_1)B_0(t_1)\\
&+2\xi_k\xi_\ell B_0^2(t_0) P_2^{ik}(t_0)P_1^{j\ell}(t_0,t_1) B_0(t_1)\\
&+\xi_k\xi_\ell  B_0^2(t_0)\, P_2^{ij}(t_0) \,\big[t_0,t_1;P_2^{k\ell}\big] B_0(t_1)\\
&-4\xi_k \xi_\ell\xi_m\xi_n B_0^3(t_0)  P_2^{ik}(t_0) P_2^{j\ell}(t_0)\,\big[t_0,t_1;P_2^{mn}\big] B_0(t_1) \,,
\end{align*}
and
\begin{align*}
B_{2,2}^{ij}(\xi,\lambda,t_0,t_1,t_2)&=
-B_0(t_0)P_{0,2}^{ij}(t_0,t_1,t_2)B_0(t_2)\\
&+\xi_k\xi_\ell B_0(t_0)P_1^{ik}(t_0,t_1)B_0(t_1)P^{j\ell}_1(t_1,t_2)B_0(t_2)\\
& -2\xi_k\xi_\ell\xi_m\xi_n P_2^{im}B_0^2(t_0) \,\big[t_0,t_1;P_2^{k\ell}\big] B_0(t_1)P^{jn}_1(t_1,t_2)B_0(t_2)\\
&+2\xi_k\xi_\ell  P_2^{ik}B_0^2(t_0) \,\big[t_0,t_1;P_1^{j\ell}(\cdot,t_2)\big]B_0(t_2)\\
&+2\xi_k\xi_\ell  P_2^{jk}B_0^2(t_0) \,\big[t_1,t_2;P_1^{i\ell}(h_{(0),\cdot})\big] B_0(t_2)\\
&+\xi_k \xi_\ell B_0(t_0)P_1^{ij}(t_0,t_1)B_0(t_1)\,\big[t_1,t_2;P_2^{k\ell}\big]B_0(t_2) \\
& -2 \xi_k \xi_\ell\xi_m\xi_n P_2^{j\ell}\,B_0^2(t_0) \, P_1^{ik}(t_0,t_1) \, B_0(t_1)\,\big[t_1,t_2;P_2^{mn}\big]\, B_0(t_2)  \\
&-2 \xi_k \xi_\ell\xi_m\xi_n B_0(t_0)P_1^{ik}(t_0,t_1)\,  P_2^{j\ell}(t_1)\,B_0^2(t_1) \,\big[t_1,t_2;P_2^{mn}\big]\, B_0(t_2)  \\
&- 2\, \xi_k\xi_\ell\xi_m\xi_n P_2^{ij}B_0^2(t_0) \,\big[t_0,t_1;P_2^{\ell m}\big]\, B_0(t_1)\, \big[t_1,t_2;P_2^{kn}\big]B_0(t_2)  \\ 
&- 4\xi_k\xi_\ell\xi_m\xi_n P_2^{ik}B_0^2(t_0) \,\big[t_0,t_1;P_2^{ j\ell}\big]\, B_0(t_1)\,\big[t_1,t_2;P_2^{mn}\big]\, B_0(t_2) \\
&+ 8\xi_k\xi_\ell\xi_m\xi_n\xi_p\xi_q P_2^{ik} P_2^{jn}B_0^3(t_0) \,\big[t_0,t_1;P_2^{\ell m}\big]\, B_0(t_1)\,\big[t_1,t_2;P_2^{pq}\big]\, B_0(t_2) \\
&+ 4\xi_k\xi_\ell\xi_m\xi_n\xi_p\xi_q P_2^{ik}B_0^2(t_0) \,\big[t_0,t_1;P_2^{\ell m}\big]\,  P_2^{jn}B_0^2(t_1)\,\big[t_1,t_2;P_2^{pq}\big]\, B_0(t_2)\\
&+2\xi_k\xi_\ell P_2^{ij} B_0^2(t_0) \,\big[t_0,t_1,t_2;P_2^{k\ell}\big] \, B_0(t_2)\\ 
&-8\xi_k \xi_\ell \xi_m\xi_n P_2^{ik} P_2^{j\ell}B_0^3(t_0) \,\big[t_0,t_1,t_2;P_2^{mn}\big] \, B_0(t_2). 
\end{align*} 
 Note that in the above formulas, $B_0(t)$ is always a scalar function and commutes with the other terms $P_i^{jk}$ which can be matrix valued functions.
Hence, we can bring all the powers of $B_0$'s in each summand up front and group them with the other scalar terms such as $\xi_j$'s.
 But, the order  of the other terms must stay untouched to respect the noncommutativity of matrix product. 
This is what we will do in the next section while we find a formula for the integral of these terms.

\subsection{Coefficients of the asymptotic expansion of the heat trace}
Using the Cauchy integral formula, we have
\begin{equation*}
e^{-tP}=\frac{-1}{2\pi i}\int_\gamma e^{-t\lambda}(P-\lambda)^{-1}d\lambda,\quad t>0.
\end{equation*}
Here $\gamma$ is  a contour around the non-negative real line with the counter clockwise orientation. 
The expansion of the symbol of the paramatrix $\sigma((P-\lambda)^{-1})$ can be used to find the asymptotic expansion of the localized heat trace  $\Tr(ae^{-tP})$ as $t\to 0^+$ for any $a\in \Ai[\dim]$:
\begin{equation*}
\Tr(ae^{-tP})\sim \sum_{n=0}^\infty c_n(a) t^{(n-\dim)/2}.
\end{equation*} 
Here, each $c_n(a)$ can be written as $\varphi(a b_n)$  for some $b_n=b_n(P)\in \Ai[\dim]$, where 
\begin{equation}\label{localterminbjs}
b_n(P)=\frac1{(2\pi)^\dim}\int_{\mathbb{R}^\dim}\frac{-1}{2\pi i}\int_\gamma b_n(\xi,\lambda)d\lambda d\xi.
\end{equation} 
For more details see e.g. \cite{Connes-Tretkoff2011} for the noncummtative tori or \cite{Gilkey1995} for a similar computation   in the commutative case.
Using the rearrangement lemma  (Lemma \ref{rearrangmentlemma}) and the fact that the contraction map and integration commute,
we can readily  see that 
\begin{equation*}
\begin{aligned}
b_2(P)=&\left(\frac1{(2\pi)^\dim}\int_{\mathbb{R}^\dim}\frac{-1}{2\pi i}\int_\gamma B_{2,1}^{ij}(\xi,\lambda,h_{(0)},h_{(1)})\, e^{-\lambda} d\lambda d\xi\right)\big(\delta_i\delta_j(h)\big)\\
&+\left(\frac1{(2\pi)^\dim}\int_{\mathbb{R}^\dim}\frac{-1}{2\pi i}\int_\gamma B_{2,2}^{ij}(\xi,\lambda,,h_{(0)},h_{(1)},,h_{(2)})\, e^{-\lambda} d\lambda d\xi\right)\big(\delta_i(h)\cdot\delta_j(h)\big).
\end{aligned}
\end{equation*} 
The dependence of $B_{2,k}^{ij}$ on  $\lambda$ comes only from different powers of $B_0$ in its terms, while its dependence on  $\xi_j$'s is the result of appearance of $\xi_j$  as well as of $B_0$ in the terms. 
Therefore, the contour integral will only contain $e^{-\lambda}$ and 
product of powers of $B_0(t_j)$.
Hence, we need to deal with a certain kind of contour integral for which  we shall use the following notation and will call them {\it $T$-functions}:
\begin{equation}\label{Tfunctionsdef}
T_{\vec{n};\vec{\alpha}}(t_0,\cdots,t_n):=\frac{-1}{\pi^{\dim/2}}\int_{\mathbb{R}^{\dim}}\xi_{n_1}\cdots \xi_{n_{2|\vec{\alpha}|-4}}\frac{1}{2\pi i}\int_\gamma e^{-\lambda} B_0^{\alpha_0}(t_0)\cdots B_0^{\alpha_n}(t_n)  d\lambda d\xi,
\end{equation}
where $\vec{n}=(n_1,\cdots,n_{2|\vec{\alpha}|-4})$ and $\vec{\alpha}=(\alpha_0,\cdots,\alpha_n)$.
We study $T$-functions and their properties in Subsection \ref{Tfunctions}.
Now, we  write a formula for $b_2(P)$.
\begin{proposition}\label{b2proposition}
For a positive   Laplace type $h$-differential operator $P$ with the symbol given by \eqref{noncommutativelaplacetypesym},  the term $b_2(P)$ in the contraction form is given by
\begin{equation*}
b_2(P)=(4\pi)^{-\dim/2}\left(B_{2,1}^{ij}(h_{(0)},h_{(1)})\big(\delta_i\delta_j(h)\big)+B_{2,2}^{ij}(h_{(0)},h_{(1)},h_{(2)})\big(\delta_i(h)\cdot\delta_j(h)\big)\right),
\end{equation*}
where the functions are defined by   
\begin{align*}
B_{2,1}^{ij}(t_0,t_1)
=& -T_{;1,1}(t_0,t_1)P_{0,1}^{ij}(t_0,t_1)
+2T_{k \ell;2,1}(t_0,t_1) P_2^{ik}(t_0)P_1^{j\ell}(t_0,t_1)\\
&+T_{k \ell;2,1}(t_0,t_1)\, P_2^{ij}(t_0) \,\big[t_0,t_1;P_2^{k\ell}\big]\\
&-4T_{k\ell  m n;3,1}(t_0,t_1) P_2^{ik}(t_0) P_2^{j\ell}(t_0)\,\big[t_0,t_1;P_2^{mn}\big],
\end{align*}
and
\begin{align*}
B_{2,2}^{ij}(t_0,t_1,t_2)=
&-T_{;1,1}(t_0,t_2) P_{0,2}^{ij}(t_0,t_1,t_2)\\
&+T_{k\ell;1,1,1}(t_0,t_1,t_2)\,   P_1^{ik}(t_0,t_1)P^{j\ell}_1(t_1,t_2)\\
&-2T_{k\ell mn;2,1,1}(t_0,t_1,t_2)\, P_2^{im}(t_0)\,\big[t_0,t_1;P_2^{k\ell}\big]\, P^{jn}_1(t_1,t_2) \\
&+2T_{k\ell;2,1}(t_0,t_2)\,  P_2^{ik}(t_0)\,\big[t_0,t_1;P_1^{j\ell}(\cdot,t_2)\big]\,\\
&+2T_{k\ell;2,1}(t_0,t_2)\, P_2^{jk}(t_0)\,\big[t_1,t_2;P_1^{i\ell}(t_0,\cdot)\big]\,\\
&+T_{k\ell;1,1,1}(t_0,t_1,t_2)\, P_1^{ij}(t_0,t_1)\,\big[t_1,t_2;P_2^{k\ell}\big]\, \\
& -2T_{k\ell mn;2,1,1}(t_0,t_1,t_2)\, P_2^{j\ell}(t_0) P_1^{ik}(t_0,t_1)\,\big[t_1,t_2;P_2^{mn}\big]\,  \\
&-2T_{k\ell mn;1,2,1}(t_0,t_1,t_2)\,  P_1^{ik}(t_0,t_1)P_2^{j\ell}(t_1)\,\big[t_1,t_2;P_2^{mn}\big]\,  \\
&- 2T_{k\ell mn;2,1,1}(t_0,t_1,t_2) \,  P_2^{ij}(t_0)\,\big[t_0,t_1;P_2^{\ell m}\big]\,\big[t_1,t_2;P_2^{kn}\big]\, \\ 
&- 4T_{k\ell mn;2,1,1}(t_0,t_1,t_2)\, P_2^{ik}(t_0)\,\big[t_0,t_1;P_2^{ j\ell}\big]\,\big[t_1,t_2;P_2^{mn}\big]\, \\
&+ 8T_{k\ell mnpq;3,1,1}(t_0,t_1,t_2)\, P_2^{ik}(t_0) P_2^{jn}(t_0)\,\big[t_0,t_1;P_2^{\ell m}\big]\,\big[t_1,t_2;P_2^{pq}\big]\, \\
&+ 4T_{k\ell m n p q;2,2,1}(t_0,t_1,t_2) P_2^{ik}(t_0)\,\big[t_0,t_1;P_2^{\ell m}\big]\,  P_2^{jn}(t_1)\,\big[t_1,t_2;P_2^{pq}\big]\\
&+2T_{k\ell;2,1}(t_0,t_2)\, P_2^{ij}(t_0)\,\big[t_0,t_1,t_2;P_2^{k\ell}\big]\,\\ 
&-8T_{k\ell mn;3,1}(t_0,t_2)\, P_2^{ik}(t_0) P_2^{j\ell}(t_0)[t_0,t_1,t_2;P_2^{mn}\big].
\end{align*} \qed
\end{proposition}
The computation of the higher heat trace densities for a Laplace type $h$-operator can be similarly carried out; expecting more terms in the results.
This would give a way to generalize results obtained for the conformally flat noncommutative two torus in \cite{Connes-Fathizadeh2016}, where $b_4$ of the Laplacian $D^2$ of the Dirac operator $D$ is computed.

\subsection{\texorpdfstring{$T$}{T}-functions and their properties}\label{Tfunctions}
Evaluating $T$-functions \eqref{Tfunctionsdef}, are the only parts of formulas for $B_{2,1}^{i,j}$ and $B_{2,2}^{i,j}$ that need to be evaluated, is not always an easy task.
We give a concise integral formula for $T$-functions and study their properties in this section  and later on show that in some cases, they can be computed explicitly.  

We start with the contour integral in \eqref{Tfunctionsdef}. 
It is evident that there are functions $f_{\alpha_1,\cdots,\alpha_{n}}$ such that
\begin{equation*}
\frac{-1}{2\pi i}\int_\gamma e^{-\lambda} B_0^{\alpha_0}(t_0)\cdots B_0^{\alpha_n}(t_n)  d\lambda
=f_{\alpha_0,\cdots,\alpha_{n}}\left(\|\xi\|_{t_0}^{2},\cdots,\|\xi\|_{t_n}^{2}\right).
\end{equation*}
Here, we denoted $P_2^{ij}(t_k)\xi_i\xi_j$ by $\|\xi\|_{t_k}^2$.
Examples of such functions are
\begin{equation*}
f_{1,1}(x_0,x_1)=-\frac{e^{-x_0}}{x_0-x_1}-\frac{e^{-x_1}}{x_1-x_0},\qquad 
f_{2,1}(x_0,x_1)=-\frac{e^{-x_0}}{x_0-x_1}- \frac{e^{-x_0}}{(x_0-x_1)^2}+\frac{e^{-x_1}}{(x_1-x_0)^2}.
\end{equation*}
We can indeed find a general formula for $f_{\vec{\alpha}}$ in terms of derivatives of the divided differences of the function $\exp({-\lambda})$:
\begin{align*}
f_{\vec{\alpha}}(x_0,\cdots,x_{n})
&=\frac{-1}{2\pi i}\int_\gamma e^{-\lambda} \prod_{j=0}^{n} (x_j-\lambda)^{-\alpha_j} d\lambda\\
&=\frac{(-1)^{|\vec{\alpha}|+1}}{2\pi i}\int_\gamma e^{-\lambda} \prod_{j=0}^{n} (\lambda-x_j)^{-\alpha_j} d\lambda\\
&=\frac{(-1)^{|\vec{\alpha}|+1}}{\vec{\beta}!}{\partial^{\vec{\beta}}_x}\left(\frac{1}{2\pi i}\int_\gamma e^{-\lambda} \prod_{j=0}^n (\lambda-x_j)^{-1} d\lambda\right)\\
&=\frac{(-1)^{|\vec{\alpha}|+1}}{\vec{\beta}!}{\partial^{\vec{\beta}}_x}\left[x_0,\cdots,x_n;\lambda\mapsto e^{-\lambda}\right],
\end{align*}
where $\vec{\beta}$ is the multi-index $(\alpha_0-1,\cdots,\alpha_{n}-1)$.
In the last formula we used a  contour integral formula for divided differences of an analytic function which can be found in \cite[Section 1.7.]{Milne1951}.
Now, we can apply Hermite-Genocchi formula \eqref{HermiteGenocchi} for $f(\lambda)=\exp({-\lambda})$ and rewrite $f_{\vec{\alpha}}$ as an integral $\frac{1}{\vec{\beta}!}\int_{\Sigma_n}e^{-\left(\sum_{j=0}^n s_j x_j\right)} \prod_{j=0}^ns_j^{\alpha_j-1} ds$ over the standard simplex.
We have
\begin{align*}
T_{\vec{n};\vec{\alpha}}(t_0,\cdots,t_n)
&=\pi^{-\frac{\dim}2}\int_{\mathbb{R}^{\dim}}\xi_{n_1}\cdots \xi_{n_{2|\vec{\alpha}|-4}} f_{\alpha_0,\cdots,\alpha_{n}}\left(\|\xi\|_{t_0}^{2},\cdots,\|\xi\|_{t_n}^{2}\right) d\xi\\
&=\pi^{-\frac{\dim}2}\int_{\mathbb{R}^{\dim}}\xi_{n_1}\cdots \xi_{n_{2|\vec{\alpha}|-4}} \frac{1}{\vec{\beta}!}\int_{\Sigma_n}e^{- \langle\sum_{j=0}^n s_j P_2(t_j)\xi,\xi\rangle} \prod_{j=0}^ns_j^{\alpha_j-1} ds d\xi\\
&=\frac{\pi^{-\frac{\dim}2}}{\vec{\beta}!}\int_{\Sigma_n}\prod_{j=0}^ns_j^{\alpha_j-1} \int_{\mathbb{R}^{\dim}}\xi_{n_1}\cdots \xi_{n_{2|\vec{\alpha}|-4}} e^{- \langle\sum_{j=0}^n s_j P_2(t_j)\xi,\xi\rangle}  d\xi ds \\
&=\frac{1}{2^{|\vec{\alpha}|-2}\vec{\beta}!}\int_{\Sigma_n}\prod_{j=0}^ns_j^{\alpha_j-1}\, \frac{\underset{\vec{n}}{\sum \prod} \left( \sum_{j=0}^n s_j P_2(t_j) \right)^{-1}_{n_in_{\sigma(i)}}}{\sqrt{\det\left( \sum_{j=0}^n s_j P_2(t_j) \right)} }  ds. 
\end{align*}
In the last step, we applied the  Wick's theorem  (See e.g. \cite{Lando-zvonkin2004}) for the inner integral.
We use $\sum \prod$ to denoted
$\sum_{\sigma\in \Pi_{|\vec{\alpha}|-2}}\prod_{i\in S/\sigma},$ 
where $\Pi_{k}$ is the set of all pairings on $S=\{1,\cdots,2k\}$.
For instance, 
\begin{equation*}
\underset{(k,\ell)}{\sum \prod} A_{n_in_{\sigma(i)}}=A_{k\ell}\quad \text{ and }\quad 
\underset{(k,\ell,m,n)}{\sum \prod} A_{n_in_{\sigma(i)}}=A_{k\ell}A_{mn}+A_{km}A_{\ell n}+A_{k n}A_{\ell m}.
\end{equation*}
Putting it all together, we have proved the following lemma.
\begin{lemma}\label{integralformofTna} 
Let $P_2(t)$ be a positive definite $\dim\times \dim$ matrix of smooth real functions and let   $B_0(t_0)=(P_2^{ij}(t_0)\xi_i\xi_j-\lambda)^{-1}$. 
Then
\begin{equation*}
T_{\vec{n};\vec{\alpha}}(t_0,\cdots,t_n)=\frac{1}{2^{|\vec{\alpha}|-2}\vec{\beta}!}\int_{\Sigma_n}\prod_{j=0}^ns_j^{\alpha_j-1} \, \frac{\underset{\vec{n}}{\sum \prod} P^{-1}(s)_{n_in_{\sigma(i)}}}{\sqrt{\det P(s)}}  ds, 
\end{equation*}
where $P(s)=\sum_{j=0}^n s_j P_2(t_j)$  and $\vec{\beta}=(\alpha_0-1,\cdots,\alpha_{n}-1)$.\qed
\end{lemma}
For our purposes, we need only two types of $T$-functions:
\begin{equation*}
\begin{aligned}
T_{\vec{n};\alpha,1}&=\frac{1}{2^{\alpha-1}(\alpha-1)!}\int_0^1 s^{\alpha-1}\frac{\underset{\vec{n}}{\sum \prod} P^{-1}(s)_{n_in_{\sigma(i)}}}{\sqrt{\det P(s)}}  ds,\\
T_{\vec{n};\alpha_0,\alpha_1,1}&=
 \frac{1}{2^{\alpha_0+\alpha_1-1}\Gamma(\alpha_0)\Gamma(\alpha_1)}
\int_0^1\hspace{-0.2cm}\int_0^{\scriptstyle 1-s_1} s_0^{\scriptscriptstyle \alpha_0-1}s_1^{\scriptscriptstyle \alpha_1-1}
\frac{\underset{\vec{n}}{\sum \prod} P^{-1}(s_0,s_1)_{n_in_{\sigma(i)}}}{\sqrt{\det P(s_0,s_1)}} ds_0ds_1.
\end{aligned}
\end{equation*}

\begin{remark}\label{basicpropTfunctions}
We list some basic properties of $T$-functions below:
\begin{enumerate}
\item[$i)$] $T_{n_{\sigma(1)},\cdots,n_{\sigma(k)};\vec{\alpha}}=T_{\vec{n};\vec{\alpha}},\qquad \sigma\in S(\vec{n}),\quad \sigma\in S_{k},$
\item[$ii)$] $T_{\vec{n};\alpha_{\sigma(0)},\alpha_{\sigma(1)},\cdots,\alpha_{\sigma(n)}}(t_{\sigma(0)},t_{\sigma(1)},\cdots,t_{\sigma(n)})=T_{\vec{n};\alpha_{0},\alpha_{1},\cdots,\alpha_{n}}(t_{0},t_{1},\cdots,t_{n}), \quad \sigma\in S_{n+1}$
\item[$iii)$] $\displaystyle{\lim_{t_n\to t_0}}T_{\vec{n};\alpha_{0},\cdots,\alpha_{n-1},1}(t_{0},t_{1},\cdots,t_{n})=T_{\vec{n};\alpha_{0}+1,\cdots,\alpha_{n-1}}(t_{0},t_{1},\cdots,t_{n-1}),$
\item[$iv)$] $T_{\vec{n};\alpha_{0},\cdots,\alpha_{j-1},0,\alpha_{j+1},\cdots,\alpha_{n}}(t_{0},\cdots\hspace{-0.05cm},t_{n})=T_{\vec{n};\alpha_{0},\cdots,\alpha_{j-1},\alpha_{j+1},\cdots,\alpha_{n}}(t_{0},\cdots\hspace{-0.05cm},t_{j-1},t_{j+1},\cdots\hspace{-0.05cm},t_{n}),$
\item[$v)$] $\frac{\partial}{\partial t_j} T_{\vec{n};\alpha_{0},\cdots,\alpha_{n}}(t_{0},\cdots,t_{n})= -\alpha_j P'(t_j)^{k,\ell}T_{\vec{n},k,\ell;\alpha_{0},\cdots,\alpha_{j}+1,\cdots,\alpha_{n}}(t_{0},\cdots,t_{n}).$
\end{enumerate} 
\end{remark}

Finally, we would like to point out a recursive formula for $T$-functions.
 \begin{lemma}\label{Tfunctionsrelations}
Let $\vec{n}=(n_1,\cdots,n_\ell)$ and $\vec{\tilde{n}}=(n_1,\cdots,n_\ell,m,n)$. 
 Then, we have 
\begin{equation*}
T_{\vec{\tilde{n}};\alpha,1}(t_0,t_1)(P_2(t_0)-P_2(t_1))^{mn}=\begin{cases}
 T_{\vec{n};\alpha-1,1}(t_0,t_1)-\frac{\underset{\vec{n}}{\sum \prod} g(t_0)_{n_in_{\sigma(i)}}}{2^{\alpha-2}\Gamma(\alpha)\sqrt{\det P_2(t_0)}}, & \alpha\geq 2,\\ &\\
\frac{2\underset{\vec{n}}{\sum \prod} g(t_1)_{n_in_{\sigma(i)}}}{\sqrt{\det P_2(t_1)}} -\frac{2\underset{\vec{n}}{\sum \prod} g(t_0)_{n_in_{\sigma(i)}}}{\sqrt{\det P_2(t_0)}}, & \alpha= 1,
\end{cases}
\end{equation*}
and
\begin{equation*}
T_{\vec{\tilde{n}};\alpha_0,\alpha_1,1}(t_0,t_1,t_2)(P_2(t_1)-P_2(t_2))^{mn}=\hspace{-0.1cm}\begin{cases}
 T_{\vec{n};\alpha_0,\alpha_1-1,1}(t_0,t_1,t_2)- T_{\vec{n};\alpha_0,\alpha_1}(t_0,t_1), & \alpha_1\geq 2,\\  &\\
 T_{\vec{n};\alpha_0,1}(t_0,t_2)- T_{\vec{n};\alpha_0,1}(t_0,t_1), & \alpha_1= 1.
\end{cases}
\end{equation*}
\end{lemma}
\begin{proof}
The key point here lies in the following formula which is a direct result of   Jacobi's formula: $d\det(A)=\det(A)\tr(A^{-1}dA)$. 
$$\frac{d}{ds}\left(\frac{\underset{\vec{n}}{\sum \prod} P^{-1}(s)_{n_in_{\sigma(i)}}}{\sqrt{\det(P(s))}}\right)=-\frac12 \frac{{\underset{\vec{\tilde{n}}}{\sum \prod} P^{-1}(s)_{n_in_{\sigma(i)}}}}{\sqrt{\det(P(s))}}(P_2(t_0)-P_2(t_1))^{mn}.$$
Then, an integration by parts will complete the proof:
\begin{align*}
&T_{\vec{\tilde{n}};\alpha,1}(t_0,t_1)(P_2(t_0)-P_2(t_1))^{mn}\\
&=\frac{1}{2^{\alpha-1}\Gamma(\alpha)}\int_0^1 s^{\alpha-1}\frac{\underset{\vec{\tilde{n}}}{\sum \prod} P^{-1}(s)_{n_in_{\sigma(i)}}}{\sqrt{\det P(s)}}  ds (P_2(t_0)-P_2(t_1))^{mn}\\
&=\frac{-1}{2^{\alpha-2}\Gamma(\alpha)}\int_0^1 s^{\alpha-1}\frac{d}{ds}\left(\frac{ \underset{\vec{n}}{\sum \prod} P^{-1}(s)_{n_in_{\sigma(i)}}}{\sqrt{\det P(s)}}\right)  ds \\
&=\frac{-1}{2^{\alpha-2}\Gamma(\alpha)}\left(\left. s^{\alpha-1}\frac{ \underset{\vec{n}}{\sum \prod} P^{-1}(s)_{n_in_{\sigma(i)}}}{\sqrt{\det P(s)}}\right]_0^1-(\alpha-1)\int_0^1 s^{\alpha-2}\frac{ \underset{\vec{n}}{\sum \prod} P^{-1}(s)_{n_in_{\sigma(i)}}}{\sqrt{\det P(s)}}  ds\right).
\end{align*}
Similar steps works to prove the second equation.
\end{proof}
This lemma can be used to further simplify the expression   for $B_{2,1}^{i,j}$ and $B_{2,2}^{i,j}$ given in Proposition \ref{b2proposition}.
As an example, we use it here to simplify the third term of $B_{2,1}^{i,j}$:
\begin{align*}
T_{k \ell;2,1}(t_0,t_1)\, P_2^{ij}(t_0) \,\big[t_0,t_1;P_2^{k\ell}\big]
=&(T_{;1,1}(t_0,t_1)-T_{;2}(t_0))\, P_2^{ij}(t_0)\frac1{t_0-t_1}\\
=&T_{;1,1}(t_0,t_1) P_2^{ij}(t_0)\frac1{t_0-t_1}-\frac{1}{\sqrt{\det P_2(t_0)}}\, P_2^{ij}(t_0)\frac1{t_0-t_1}.
\end{align*}
Applying the lemma on the other terms, we can find a new and simpler formula for $B_{2,1}^{i,j}$ given by
\begin{align*}
B_{2,1}^{ij}(\xi,\lambda,t_0,t_1)
=
&T_{;1,1}(t_0,t_1)\left(-P_{0,1}^{ij}(t_0,t_1)+ P_2^{ij}(t_0)\frac1{t_0-t_1}\right)\\
&+2T_{k \ell;2,1}(t_0,t_1)  P_2^{ik}(t_0)\left(P_1^{j\ell}(t_0,t_1)-2 P_2^{j\ell}(t_0)\,\frac1{t_0-t_1}\right).
\end{align*}
In this new form, each term of expression may not be smooth anymore and this is the reason that we won't write down the simplified formula for $B_{2,2}^{ij}$.
However, we shall use this lemma in Section \ref{Totalcurvaturesection} where we want to give a formula for the total curvature of the curved noncommutative tori.

We will, in some cases, evaluate $T$-functions.
For conformally flat metrics,  we shall find a closed formula for $T$-functions in terms of divided differences (Lemma \ref{Tfunctionsconfomallyflatlemma}).
A similar formula for the twisted product metrics is computed in Lemma \ref{Talphaklemma}. 
We also have a formula for some of the $T$-functions in dimension two (see  Lemma \ref{Tfunctionsdimtwo}).
\begin{example}\label{doubalytsitedtf}
In this example, we find some $T$-functions for a   Laplace type $h$-differential operator on $\nctorus[4]$ with the principal symbol given by entries of the matrix
\begin{equation*}
f(t)g^{-1}\oplus \tilde{f}(t)\tilde{g}^{-1},
\end{equation*} 
where $g$ and $\tilde{g}$ are  two by two positive definite symmetric real matrices and $f$ and $\tilde{f}$ are  positive smooth functions.
First we note that
\begin{equation*}
\sqrt{\det P(s)}=\frac1{\sqrt{|g||\tilde{g}|}}(s f(t_0)+(1-s)f(t_1))(s \tilde{f}(t_0)+(1-s)\tilde{f}(t_1)).
\end{equation*} 
The integral formula from Lemma \ref{integralformofTna} for $T_{\vec{n};\alpha,1}$ turns out to be a basic integral and we have
\begin{align*}
T_{;1,1}(t_0,t_1)
=& \frac{\sqrt{|g||\tilde{g}|}}{f(t_0)\tilde{f}(t_1)-f(t_1)\tilde{f}(t_0)}\log\left(\frac{f(t_0)\tilde{f}(t_1)}{f(t_1)\tilde{f}(t_0)}\right),
\end{align*}
and for $1\leq k, l\leq 2$, we have
\begin{align*}
&T_{k,l;2,1}(t_0,t_1)=
\frac{g_{kl}\sqrt{|g||\tilde{g}|} \left(f(t_1) \tilde{f}(t_0)+f(t_0) \tilde{f}(t_1) \left(\log \left(\frac{f(t_0) \tilde{f}(t_1)}{f(t_1) \tilde{f}(t_0)}\right)-1\right)\right)}{2 f(t_0) \left(f(t_1) \tilde{f}(t_0)-f(t_0) \tilde{f}(t_1)\right)^2}.
\end{align*}
The functions $T_{k,l;2,1}$ for $3\leq k, l\leq 4$ are given by the same formula after we interchange $f$ with $\tilde{f}$ and ${g}$ with $\tilde{g}$.
These $T$-functions can be computed in the same manner in cases with more than two metrics and functions, called multiply twisted product metrics of two tori. 
\end{example}

Further exploration of the properties of $T$-functions  is an important problem.
In particular, finding an explicit formula of $T$-functions for more general functional metrics is an open problem.
Along these lines, The following point may be noticed.
 The $T$-functions for conformally flat metrics are given by derivatives of divided differences which appear as coefficient of the polynomial interpolation. 
Similarly, integrals over the standard simplex, such as the one found in Lemma \ref{Tfunctions} for $T$-functions, appear as the coefficeints of a multivariate polynomial interpolation called Kergin interpolation (see e.g. \cite{Micchelli1980}).

\section{Functional metrics and their scalar curvature density}

In this section, we  apply  Proposition \ref{b2proposition} to a geometric operator on the noncommutative tori equipped with  a new class of metrics,  called   functional metrics, which are introduced in this work for the first time.
First, we construct the Laplacian on functions of $(\nctorus[\dim],g)$ and find an $h$-differential operator $\lap_{0,g}$ which is antiunitary equivalent to the Laplacian. 
We then write the symbol of  $\lap_{0,g}$ in the contraction form and apply Proposition \ref{b2proposition} to find $b_2(\lap_{0,g})$.

The geometric importance of the second density of the heat trace of the Laplacian on functions (or Laplacian of the Dirac operator)  originates from the  fact in spectral geometry that  on manifolds \cite[Section]{Gilkey1995}  this density is a multiple of the scalar curvature of the metric.
This fact provided a way to define the scalar curvature $\scalar$, called scalar curvature density, for noncommutative spaces;
an idea which was initiated by  \cite{Connes-Moscovici2014,Fathizadeh-Khalkhali2013} see also \cite[Section 11]{Connes-Marcolli2008}.
Here, we define the {\it scalar curvature density} of a given functional metric to be 
\begin{equation}\label{scalarcurvaturedef}
\scalar=(4\pi)^{\frac\dim2}b_2(\lap_{0,g}).
\end{equation}    
We compute the scalar curvature density for two classes of examples in all dimensions: conformally flat metrics  and twisted product flat metrics.

\subsection{Functional metrics and their Laplacian}\label{functionalmetricsec}
We first introduce a new class of metrics on the noncommutative tori and then a   Laplace type $h$-differential operator to study the spectral invariants of the geometries defined by these metrics.
 
\begin{definition}
Let $h$ be a selfadjoint smooth element of the noncommutative $\dim$-tori
and let $g_{ij}:\mathbb{R}\to \mathbb{R}$, for $1\leq i,j\leq \dim$, be smooth functions such that the matrix $\big(g_{ij}(t)\big)$ is a positive definite matrix for every $t$ in a neighborhood of the spectrum of $h$.
We shall refer to $g_{ij}(h)$ as a {\it functional metric} on $\A[\dim]$.
\end{definition}
\noindent The condition is  equivalent to saying that  $g_{ij}(h)\xi^i\xi^j\in \A[\dim]^+$ for every $(\xi^1,\cdots,\xi^{\dim})\in \mathbb{R}^{\dim}$.

Without loss of generality, we always assume that functions $g_{ij}$ are Schwartz class functions. 
Since  the spectrum of $h$ is a compact subset of $\mathbb{R}$, we can always replace $g_{ij}(t)$ with Schwartz or even compactly supported functions $\tilde{g}_{ij}$ such that $\tilde{g}_{ij}(h)=g_{ij}(h)$. 
In our notations we frequently use the property of functional calculus that  $F\circ g_{ij}(h)=F(g_{ij})(h)$ for any given smooth function $F:\mathbb{R}^{\dim^2}\to \mathbb{R}$. 
For example, we shall denote  the determinant of $(g_{ij}(h))$ by either $|g(h)|$ or $|g|(h)$.
Also, since the matrix $(g_{ij}(x))$ is invertible on a neighborhood of the spectrum of $h$,
we use $g^{ij}(h)$ to denote the entries of $(g_{ij}(h))^{-1}$.

To construct the Laplacian on $\nctorus[\dim]$ equipped with a functional metric $g_{ij}(h)$, we first consider the exterior derivative $\delta:\Ai[\dim]\to \Omega^1\nctorus[\dim]=\Ai[\dim]\otimes \mathbb{C}^{\dim}$ which is defined by
\begin{equation*}
\delta(a)=\left( i\delta_1(a),  i\delta_2(a),\cdots, i\delta_\dim(a)\right).
\end{equation*} 
The key role of the functional metric $g_{ij}(h)$ then is in turning $\Ai[\dim]$ and $\Omega^1\nctorus[\dim]$ into Hilbert spaces and finding the adjoint of $\delta$.
Analogous to the classical theory, we define the inner product on functions $\Ai[\dim]$ to be
\begin{equation*}
\langle a,b\rangle_0=\varphi(b^*a\sqrt{|g|}(h)),\quad a,b\in\Ai[\dim].
\end{equation*}
In addition, the inner product on 1-forms $\Omega^1\nctorus[\dim]$ is defined to be
\begin{equation*}
\left\langle \vec{a}, \vec{b}\right\rangle_1=\varphi\left(g^{ij}(h)b_i^*a_j\sqrt{|g|}(h)\right).
\end{equation*}
Here $\vec{a}=(a_1,a_2,\cdots,a_\dim)$, $\vec{b}=(b_1,b_2,\cdots,b_\dim)$ are elements of $\Omega^1 \nctorus[\dim]$. 
The Hilbert space completion of functions $\Ai[\dim]$ and 1-forms $\Omega^1\nctorus[\dim]$ with respect to the corresponding inner products will be denoted by $\mathcal{H}_{0,g}$ and $\mathcal{H}_{1,g}$ respectively.

The exterior derivative, as a densely defined operator from $\mathcal{H}_{0,g}$ to $\mathcal{H}_{1,g}$, is a closed unbounded operator from $\mathcal{H}_0$ to $\mathcal{H}_1$ and its formal adjoint $\delta^* :\mathcal{H}_{1,g}\to\mathcal{H}_{0,g}$ on elements of $\Omega^1\nctorus[\dim]$ is given by
\begin{equation*}
\delta^*(\vec{b})=-i b_jg^{jk}(h)\delta_k\left(\sqrt{|g|}(h)\right)|g|^{-\frac12}(h)-ib_j\delta_k\left(g^{jk}(h)\right)-i\delta_k(b_j)g^{jk}(h).
\end{equation*} 
Therefore, $\delta^*\delta:\mathcal{H}_{0,g}\to \mathcal{H}_{0,g}$ on elements of $\Ai[\dim]$ is given by
\begin{equation*}
\delta_j(a)g^{jk}(h)\delta_k(|g|^{\frac12}(h))|g|^{-\frac12}(h)+\delta_j(a)\delta_k(g^{jk}(h))+i\delta_k(\delta_j(a))g^{jk}(h).
\end{equation*}
To carry the spectral analysis of the Laplacian $\delta^*\delta:\mathcal{H}_{0,g}\to \mathcal{H}_{0,g}$, we switch to the antiunitary equivalent setting as follows.
Let $\mathcal{H}_0$ be the Hilbert space obtained by the GNS construction from $\A[\dim]$ using the nonperturbed tracial state $\varphi$. 
\begin{proposition}\label{laplacianforg}
The operator $\delta^*\delta:\mathcal{H}_{0,g}\to \mathcal{H}_{0,g}$ is antiunitary equivalent to  a  Laplace type $h$-differential operator $\lap_{0,g}:\mathcal{H}_0\to \mathcal{H}_0$ whose symbol, when expressed in the contraction form, has the functional parts given by 
\begin{align*}
&P_2^{jk}(t_0)= g^{jk}(t_0),\\
&P_1^{jk}(t_0,t_1)
=|g|^{-\frac14}(t_0)\big[t_0,t_1;|g|^{\frac14}\big]g^{jk}(t_1)+\big[t_0,t_1;g^{jk}\big]+|g|^{\frac14}(t_0)g^{jk}(t_0)\big[t_0,t_1;|g|^{-\frac14}\big],\\
&P_{0,1}^{jk}(t_0,t_1)
=|g|^{\frac14}(t_0)g^{jk}(t_0)\big[t_0,t_1;|g|^{-\frac14}\big],\\
&P_{0,2}^{jk}(t_0,t_1,t_2)
=|g|^{-\frac14}(t_0)\big[t_0,t_1;g^{jk}|g|^{\frac12}\big] \big[t_1,t_2;|g|^{-\frac14}\big]+2|g|^{\frac14}(t_0)g^{jk}(t_0)\big[t_0,t_1,t_2;|g|^{-\frac14}\big].
\end{align*}
 \end{proposition}
 \begin{proof}
The right multiplication by $|g|^{-1/4}$ extends to a unitary operator $W:\mathcal{H}_0 \to \mathcal{H}_{0,g}$ and the map $a\mapsto a^*$ gives an antiunitary map on $\mathcal{H}_0$ which we denote by $J$.
Then, $\delta^*\delta$ is antiunitary equivalent to the operator $JW^* \delta^*\delta WJ$, which we denote by $\lap_{0,g}$.
Using the basic  property of $*$ and its relation with the derivations, $\delta_j(a^*)=-\delta_j(a)^*$,  we immediately obtain
\begin{equation}\label{laplacianonfunctions}
\begin{aligned}
\lap_{0,g}(a)
=&g^{jk}(h)\delta_k\delta_j(a)+\Big(|g|^{-\frac14}\delta_k\big(|g|^{\frac14}g^{jk}(h)\big) +|g|^{\frac14}g^{jk}(h)\delta_k(|g|^{-\frac14})\Big)\delta_j(a)\\
&+ |g|^{-\frac14}\delta_k\left(|g|^{\frac12}g^{jk}(h)\delta_j(|g|^{-\frac14})\right)a.
\end{aligned}
\end{equation}
The operator $\lap_{0,g}$ is clearly a differential operator and the homogeneous parts of its symbol reads
\begin{align*}
p_2(\xi)
&= g^{jk}(h)\xi_j\xi_k,\\
p_1(\xi)
&=\left(|g|^{-\frac14}\delta_k(|g|^{\frac14})g^{jk}(h)+\delta_k(g^{jk}(h))+|g|^{\frac14}g^{jk}(h)\delta_k(|g|^{-\frac14})\right)\xi_j,\\
p_0(\xi)&= |g|^{-\frac14}\delta_k(|g|^{\frac12})g^{jk}(h)\delta_j(|g|^{-\frac14})+|g|^{\frac14}\delta_k(g^{jk}(h))\delta_j(|g|^{-\frac14})+|g|^{\frac14}g^{jk}(h)\delta_k\delta_j(|g|^{-\frac14}). 
\end{align*}
We then apply Theorem \ref{derivationoffunctionofh} and  express the above terms in the contraction form.
\end{proof}
\begin{remark}
The analysis of functional metrics on manifolds  is an interesting topic to be further investigated.
In particular, classifying Riemannian metrics that can be written locally or globally as a functional metric, with $h$ being a smooth function on the manifold, seems a nice problem. \\
The scalar curvature of a functional metric $g_{ij}(h)dx^idx^j$ on an open subset $U$ of $\mathbb{R}^\dim$  is given by
\begin{equation*}
\begin{aligned}
R
=&\Big(
-g^{-1}\tr\left(g^{-1}g''\right)
-\frac{1}{4} g^{-1}\tr\left(g^{-1}g'\right)^2+\frac{3}{4} g^{-1}\tr\left(g^{-1}g'g^{-1}g'\right)
+\tr\left(g^{-1}g'\right) g^{-1}g'g^{-1}\\
&\,\,
+g^{-1}g''g^{-1}
-\frac{3}{2} g^{-1}g'g^{-1}g'g^{-1}\Big)^{\mu \nu }\frac{\partial h}{\partial x^\mu}\frac{\partial h}{\partial x^\nu}
 +\Big(g^{-1}g'g^{-1}-g^{-1}\tr\left(g^{-1}g'\right)\Big)^{\mu \nu }\frac{\partial^2 h}{\partial x^\mu \partial x^\nu}.
\end{aligned}
\end{equation*}
\end{remark}

\subsection{Conformally flat metrics}
An important case of the functional   metric  is the conformally flat metric
\begin{equation}\label{conflatmetric}
g_{ij}(t)=f(t)^{-1}g_{ij},
\end{equation}
 where $f$ is a positive smooth function and $g_{ij}$'s are the entries of a constant metric on $\mathbb{R}^\dim$.
The formula \eqref{laplacianonfunctions} for the operator $\lap_{0,g}$  is simplified for  conformally flat metric in general dimension $d$
\begin{equation*}
\begin{aligned}
\lap_{0,g}(a)
&=f(h)g^{jk}\delta_k\delta_j(a)
+g^{jk}\Big(f(h)^{\frac{\dim}4}\delta_k\big(f(h)^{1-\frac{\dim}4}\big)+f(h)^{1-\frac{\dim}4}\delta_k\big(f(h)^{\frac{\dim}4}\big)\Big)\delta_j(a)\\
&+ g^{jk}f(h)^{\frac{\dim}4}\delta_k\left(f(h)^{1-\frac{\dim}2}\delta_j(f(h)^{\frac{\dim}4})\right)a,
\end{aligned}
\end{equation*}
The functions given by Proposition \ref{laplacianforg}, for the conformally flat metrics, gives us the following:
\begin{equation}\label{symfunconformal}
\begin{aligned}
&P_2^{jk}(t_0)= g^{jk}f(t_0),\\
&P_1^{jk}(t_0,t_1)
=g^{jk}\Big(f(t_0)^{\frac{\dim}4}\big[t_0,t_1;f^{1-\frac{\dim}4}\big]+f(t_0)^{1-\frac{\dim}4}\big[t_0,t_1;f^{\frac{\dim}4}\big]\Big),\\
&P_{0,1}^{jk}(t_0,t_1)=g^{jk}f(t_0)^{1-\frac{\dim}4}\big[t_0,t_1;f^{\frac{\dim}4}\big],\\
&P_{0,2}^{jk}(t_0,t_1,t_2)= 
g^{jk}\Big(f(t_0)^{\frac{\dim}4}\big[t_0,t_1;f^{1-\frac{\dim}2}\big]\big[t_1,t_2;f^{\frac{\dim}4}\big]+2f(t_0)^{1-\frac{\dim}4}\big[t_0,t_1,t_2;f^{\frac{\dim}4}\big]\Big).
\end{aligned}
\end{equation}
A careful examination of  formula \eqref{symfunconformal} shows that  for any  function $P^{ij}_{\bullet}$ there exist a function $P_{\bullet}$ such that 
$P^{ij}_{\bullet}=g^{ij}P_{\bullet}.$
We have similar situation with the $T$-functions for conformally flat metrics. 
\begin{lemma}\label{Tfunctionsconfomallyflatlemma}
Let  $\vec{\alpha}$ and $\vec{n}=(n_1,\cdots,n_{2|\vec{\alpha}|-4})$ be two multi-indicies.
Then the $T$-function $T_{\vec{n,\alpha}}$ for the conformally flat metric \eqref{conflatmetric} is of the form  
$$T_{\vec{n,\alpha}}(t_0,\cdots,t_n)=\sqrt{|g|}\underset{\vec{n}}{\sum \prod}g_{n_in_{\sigma(i)}} T_{\vec{\alpha}}(t_0,\cdots,t_n).$$ 
The function $T_{\vec{\alpha}}$ in dimension $d\neq 2$ is given by 
\begin{equation}\label{Tfunctionsconfomallyflat}
T_{\vec{\alpha}}(t_0,\cdots,t_n)
=\frac{(-1)^{|\vec{\alpha}|-1}\Gamma(\frac{\dim}2-1)}{\Gamma(\frac{\dim}2+|\vec{\alpha}|-2)} \left.\partial_x^{\vec{\beta}} \big[x_0,\cdots,x_n;u^{1-\frac{d}2}\big]\right|_{x_j=f(t_j)},
\end{equation}
where $\vec{\beta}=(\alpha_0-1,\cdots,\alpha_{n}-1)$.
\end{lemma}
\begin{proof}
First note that the terms involving $P(\vec{s})=\sum_{j=0}^n s_j f(t_j)$ are given by
\begin{equation*}
P^{-1}(\vec{s})_{k\ell}=g_{k\ell}\Big(\sum_{j=0}^{n} s_j f(t_j)\Big)^{-1}, \qquad \qquad 
\sqrt{\det P(\vec{s})}=\frac1{|g|}\Big(\sum_{j=0}^{n} s_j f(t_j)\Big)^{\frac{\dim}2}.
\end{equation*}
This implies that the $T$-function $T_{\vec{n,\alpha}}$ can be written as $\sqrt{|g|}\underset{\vec{n}}{\sum \prod}g_{n_in_{\sigma(i)}} T_{\vec{\alpha}}$, where $T_{\vec{\alpha}}$ has an integral formula which we evaluate below.
Let $x_j$, for $j=0,\cdots,\leq n$, be positive and distinct real numbers. 
Then for $\dim\neq 2$, we have
\begin{equation*}
\begin{aligned}
&\int_{\Sigma_n}\prod_{j=0}^ns_j^{\alpha_j-1}\Big(\sum_{j=0}^{n} s_j x_j\Big)^{-\frac{\dim}2-|\vec{\alpha}|+2}  ds \\
&=\frac{(-1)^{|\vec{\alpha}|-n-1}\Gamma(\frac{\dim}2+n-1)}{\Gamma(\frac{\dim}2+|\vec{\alpha}|-2)}\partial_x^{\vec{\beta}}\int_{\Sigma_n}\Big(\sum_{j=0}^{n} s_j x_j\Big)^{-\frac{\dim}2-n+1}  ds\\
&=\frac{(-1)^{|\vec{\alpha}|-1}\Gamma(\frac{\dim}2-1)}{\Gamma(\frac{\dim}2+|\vec{\alpha}|-2)} \partial_x^{\vec{\beta}} \big[x_0,\cdots,x_n;u^{1-\frac{d}2}\big].
\end{aligned}
\end{equation*}
Here again $\vec{\beta}=(\alpha_0-1,\cdots,\alpha_{n}-1)$.
Replacing $x_j$ by $f(t_j)$, we find the formula \eqref{Tfunctionsconfomallyflat} for functions $T_{\vec{\alpha}}$.
\end{proof}
As an example, we have
\begin{equation*}
\begin{aligned}
T_{\alpha,1}(t_0,t_1)=&\frac{(-1)^{\alpha}\Gamma(\frac{\dim}2-1)}{2^{\alpha-1}\Gamma(\frac{\dim}2+\alpha-1)}\times \\
&\Big(\frac{f(t_1)^{1-\frac{\dim}2}}{(f(t_1)-f(t_0))^{\alpha}}-\sum_{m=0}^{\alpha-1}\frac{(-1)^m \Gamma(\frac{\dim}2+m-1)}{\Gamma(\frac{\dim}2-1)m!} \frac{f(t_0)^{-\frac{\dim}2-m+1}}{(f(t_1)-f(t_0))^{\alpha-m}}\Big).
\end{aligned}
\end{equation*} 

Note that  for dimension two, $T_{\alpha,1}(t_0,t_1)$ can be obtained by taking the limit of \eqref{Tfunctionsconfomallyflat} as  $\dim$ approaches 2. 
When $f(t)=t$, we have  
\begin{equation*}
T_{\alpha,1}(t_0,t_1)=\frac{(-1)^{\alpha-1}}{2^{\alpha-1} \Gamma(\alpha)^2} \partial_{t_0}^{\alpha-1}\big[t_0,t_1;\log(u)\big].
\end{equation*}
We shall not consider dimension two separately, but just as the limit $\dim\to 2$.

\begin{remark}
It was first noted in \cite{Liu2018-II} that for conformally flat metrics, the integral that defines functions $T_{\vec{n,\alpha}}$ can be written in terms of Gauss hypergeometric ${}_2F_1$ and Appell  hypergeometric $F_1$ functions.  
More precisely, in our notation, we have
\begin{equation*}
T_{\vec{n};\alpha,1}(t_0,t_1)=\frac{\Gamma(\frac{d}{2})   \underset{\vec{n}}{\sum \prod} g_{n_in_{\sigma(i)}}}{2^{\alpha-1}(\alpha-1)! \Gamma(\frac{d}{2}+\alpha-1)  }  f(t_1)^{-\frac{\dim}2-\alpha+1} {}_2F_1^{(\alpha-1)}\Big(\frac{\dim}2,1;2;1-\frac{f(t_0)}{f(t_1)}\Big),
\end{equation*}
and
\begin{equation*}
\begin{aligned}
T_{\vec{n};\alpha_0,\alpha_1,1}(t_0,t_1,t_2)=& \frac{ \Gamma(\frac{d}2+1) \underset{\vec{n}}{\sum \prod}g_{n_in_{\sigma(i)}}}{2^{\alpha_0+\alpha_1} \Gamma(\frac{d}2+\alpha_0+\alpha_1+1)\Gamma(\alpha_0)\Gamma(\alpha_1)} \times\\
&f(t_2)^{-\frac{\dim}2-\alpha_0-\alpha_1+1} F_1^{(\alpha_0-1,\alpha_1-1)}\left( \frac{\dim}2+1;1,1;3; 1-\frac{f(t_0)}{f(t_2)},1-\frac{f(t_1)}{f(t_2)} \right).
\end{aligned}
\end{equation*}
What we found in  \eqref{Tfunctionsconfomallyflat} is that these hypergeometric  functions for the specific parameters, can be written in terms of divided differences of the function $u^{1-\frac{\dim}2}$.
\end{remark}

Finally, we can substitute \eqref{symfunconformal} and $T$-functions \eqref{Tfunctionsconfomallyflat} in the formulas from Proposition \ref{b2proposition} and obtain the curvature of conformally flat noncommutative torus of dimension $\dim$. 
\begin{theorem}\label{scalarcrvatureconformal}
The scalar curvature of the $\dim$-dimensional noncommutative tori $\mathbb{T}_\theta^\dim$ equipped with the metric $f(h)^{-1}g_{ij}$ is given by
\begin{equation*}
R=\sqrt{|g|} \Big(K_{\dim}(h_{(0)},h_{(1)})(\triangle(h))+H_{\dim}(h_{(0)},h_{(1)},h_{(2)})(\Box(h))\Big),
\end{equation*}
where  $\lap(h)=g^{ij}\delta_i\delta_j(h)$, $\Box(h)=g^{ij}\delta_i(h)\cdot\delta_j(h)$. 
The functions $K_d$ and $H_d$ are given by
\begin{equation}\label{HKfunctions}
\begin{aligned}
K_{\dim}(t_0,t_1)=&K_{\dim}^t(f(t_0),f(t_1))\big[t_0,t_1;f\big],\\
H_{\dim}(t_0,t_1,t_2)=&H_{\dim}^t(f(t_0),f(t_1),f(t_2))\big[t_0,t_1;f\big]\big[t_1,t_2;f\big]\\
&+2K_{\dim}^t(f(t_0),f(t_2))\big[t_0,t_1,t_2;f\big],
\end{aligned}
\end{equation}
Where $K^t_d$ and $H^t_d$ are the functions $K_d$ and $H_d$ when $f(t)=t$.
For $\dim\neq 2$, they can be computed to be
\begin{align*}
K_{\dim}^t(x,y)=\frac{4\, x^{2-\frac{3 \dim }{4}} y^{2-\frac{3 \dim }{4}} }{\dim(\dim -2) (x-y)^3}\left((\dim -1) x^{\frac\dim 2} y^{\frac{\dim }{2}-1}-(\dim -1) x^{\frac{\dim }{2}-1} y^{\frac \dim 2}-x^{\dim -1}+y^{\dim -1}\right),
\end{align*}
and
\begin{align*}
H_d^t(x,y,z)=&\frac{2 x^{-\frac{3 \dim }{4}} y^{-\dim } z^{-\frac{3 \dim }{4}}}{(\dim -2) \dim  (x-y)^2 (x-z)^3 (y-z)^2}\times\\
 \Big(
 & x^{\dim } y^{\dim } z^2(x-y) \left(3 x^2 y-2 x^2 z-4 x y^2+4 x y z-2 x z^2+y z^2\right)
\\
 &+x^{\dim } y^{\frac{\dim }{2}+1} z^{\frac{\dim }{2}+1} (x-z)^2 (z-y) (\dim  x+(1-\dim ) y)
\\
 &
 +x^{\dim }y^3 z^{\dim } (z-x)^3
 + x^{\frac{\dim }{2}+1} y^{\frac{3 \dim }{2}}z^2 (x-y) (x-z)^2
\\
 &
 +2 (\dim -1) x^{\frac{\dim }{2}+1} y^{\dim } z^{\frac{\dim }{2}+1} (x-y) (x-z) (z-y) (x-2 y+z)
\\
 &-x^{\frac{\dim }{2}+1} y^{\frac{\dim }{2}+1} z^{\dim } (x-y) (x-z)^2 ((1-\dim ) y+\dim  z)
 -x^2 y^{\frac{3 \dim }{2}} z^{\frac{\dim }{2}+1} (x-z)^2 (z-y)
\\
 &
 +x^2 y^{\dim } z^{\dim } (y-z) \left(x^2 y-2 x^2 z+4 x y z-2 x z^2-4 y^2 z+3 y z^2\right)\Big).
\end{align*}
These functions for the dimension two are given by
\begin{align*}
K_2^t(x,y)=&-\frac{\sqrt{x} \sqrt{y} }{(x-y)^3}((x+y) \log (x/y)+2(y-x)),\\
H_2^t(x,y,z)=&\frac{2 \sqrt{x} \sqrt{z}}{(x-y)^2 (x-z)^3 (y-z)^2}\times\\
&\Big(-(x-y) (x-z) (y-z) (x-2 y+z) +y (x-z)^3 \log (y)  \\
&\,\,+(y-z)^2 (-2 x^2+x y+y z)\log (x) -(x-y)^2 (x y+zy-2 z^2)\log (z))\Big).
\end{align*}
\end{theorem}
\begin{proof}
A closer look at the formulas of  Proposition \ref{b2proposition} reveals that for conformally flat metrics, the dependence of $B_{2,1}^{ij}$ and $B_{2,2}^{ij}$ on the indices is simply through a factor of $g^{ij}$.
In other words, there exist  functions $K_{\dim}(t_0,t_1)$ and $H_{\dim}(t_0,t_1,t_2)$ such that  
\begin{equation*}
\begin{aligned}
B_{2,1}^{ij}(h_{(0)},h_{(1)})(\delta_i\delta_j(h))
&=\sqrt{|g|}g^{ij} K_{\dim}(h_{(0)},h_{(1)})(\delta_i\delta_j(h))\\
&= \sqrt{|g|}K_{\dim}(h_{(0)},h_{(1)})(g^{ij}\delta_i\delta_j(h)),\\
B_{2,2}^{ij}(h_{(0)},h_{(1)},h_{(2)})(\delta_i(h)\cdot\delta_j(h))
&=\sqrt{|g|} H_{\dim}(h_{(0)},h_{(1)},h_{(2)})(\delta_i(h)\cdot\delta_j(h))\\
&= \sqrt{|g|} K_{\dim}(h_{(0)},h_{(1)},h_{(2)})(g^{ij}\delta_i(h)\cdot\delta_j(h)).
\end{aligned}
\end{equation*}
To reach this result, we applied identities  such as
\begin{align*}
&\underset{(k,\ell)}{\sum \prod} g_{n_in_{\sigma(i)}}g^{k\ell}g^{ij}=\dim g^{ij},
&&\underset{(k,\ell,m,n)}{\sum \prod} g_{n_in_{\sigma(i)}}g^{ij}g^{\ell m}g^{kn}=(\dim^2+2\dim)g^{ij},\\
&\underset{(k,\ell,m,n)}{\sum \prod} g_{n_in_{\sigma(i)}}g^{ik}g^{j\ell}g^{mn}=(\dim+2)g^{ij},
&&\underset{(k,\ell,m,n,p,q)}{\sum \prod} g_{n_in_{\sigma(i)}}g^{ik}g^{jn}g^{\ell m}g^{pq}=(\dim^2+6\dim+8)g^{ij}.
\end{align*}
To continue it is enough to take $f(t)=t$ and we denote the functions found for $f(t)=t$ by $K_{\dim}^t$ and $H_{\dim}^t$. 
Note that for $f(t)=t$ we have $[t_0,t_1;P^{ij}_2]=1$ and $[t_0,t_1,t_2;P^{ij}_2]=0$ which simplify the formula.
Then for a general function $f$, we simply take $\tilde{h}=f(h)$ and then the result is given as 
$$K_{\dim}^t(\tilde{h}_{0},\tilde{h}_{1})(g^{ij}\delta_i\delta_j(\tilde{h}))+H_{\dim}^t(\tilde{h}_{0},\tilde{h}_{1},\tilde{h}_{2})(g^{ij}\delta_i(\tilde{h})\delta_j(\tilde{h})).$$
By Theorem \ref{derivationoffunctionofh} and equation \eqref{secondderivationoffunctionofh}, we can turn everything in terms of original element $h$ and get the formulas given in \eqref{HKfunctions}.

To find the functions $K^t_{2}$ and $H^t_{2}$, as mentioned before,  we should evaluate the limit of $K^t_{\dim}$ and $H^t_{\dim}$ as $\dim\to 2$.
\end{proof}
Note that the Function $K_\dim^t(x,y)$ is the symmetric part of the function
\begin{equation*}
\frac{8 x^{2-\frac{ \dim }{4}} y^{2-\frac{3 \dim }{4}}}{\dim(\dim -2)   (x-y)^3} \left((\dim -1)  y^{\frac{\dim }{2}-1}-x^{\frac{\dim}2 -1}\right).
\end{equation*}
Similarly, $H_\dim^t(x,y,z)$ is equal to $(F_d(x,y,z)+F_d(z,y,x))/2$ where
\begin{align*}
F_\dim(x,y,z)=
&\frac{4 x^{-\frac{3 \dim }{4}} y^{-\dim } z^{-\frac{3 \dim }{4}}}{\dim(\dim -2)   (x-y)^2 (x-z)^3 (y-z)^2}\times
\\
 &\Big(
 x^{\dim } y^{\dim } z^2 (x-y) (3 x^2 y-2 x^2 z-4 x y^2+4 x y z-2 x z^2+y z^2)
\\
 &\, \, +x^{\dim } y^{\frac{\dim }{2}+1} z^{\frac{\dim }{2}+1} (x-z)^2 (z-y) (\dim  x+(1-\dim ) y)
 +\frac{1}{2}  x^{\dim } y^3 z^{\dim } (z-x)^3
\\
  &\,\, +x^{\frac{\dim }{2}+1} y^{\frac{3 \dim }{2}} z^2 (x-y) (x-z)^2
 +2 (\dim -1) x^{\frac{\dim }{2}+1} y^{\dim } z^{\frac{\dim }{2}+1} (x-y)^2 (x-z) (z-y)
\Big).
\end{align*}

In low dimensions, the curvature of the conformally flat metrics was studied in \cite{Fathizadeh-Khalkhali2013,Connes-Moscovici2014,Khalkhali-Motadelro-Sadeghi2016,Fathizadeh-Khalkhali2015,Dong-Ghorbanpour-Khalkhali2018}.
Here we show that  those results can be recovered from our general formula.
We should first note that the functions found in all the aforementioned works are written in terms of $[h,\cdot]$, denoted by $\Delta$, rather than $h$ itself.
To produce those functions from our result, we are only required  to  apply a linear  substitution on the variables $t_j$ in terms of new variables $s_j$ (cf. \cite{Lesch2017}).
On the other hand, it is important to  note that the functions $K_d^t(x,y)$ and $H_d^t(x,y,z)$ are homogeneous rational functions  of order $-\frac{d}2$ and  $-\frac{\dim}2-1$ respectively.
Using formula \eqref{HKfunctions}, it is clear that the functions $K_d(t_0,t_1)$ and $H_d(t_0,t_1,t_2)$ are homogeneous of order $1-\frac{\dim}2$ in $f(t_j)$'s.
This is the reason that for function $f(t)=e^t$  and a linear substitution such as  $t_j=\sum_{m=0}^js_m$, a factor of some power of $e^{s_0}$ comes out.
This term can be replaced by a power of $e^h$ multiplied from the left to the final outcome.
This explains how the functions in the aforementioned papers have one less variable than our functions.
In other words, we have
\begin{equation*}
K_{\dim}(s_0,s_0+s_1)=e^{(1-\frac{d}{2})s_0}K_{\dim}(s_1),\qquad 
H_{\dim}(s_0,s_0+s_1,s_0+s_1+s_2)=e^{(1-\frac{d}{2})s_0}H_{\dim}(s_1,s_2). 
\end{equation*} 
For instance, function $K_d(s)$ is given by
\begin{equation*}
K_{\dim}(s_1)=\frac{8 e^{\frac{\dim+2}{4} s_1} \left((\dim-1) \sinh \left(\frac{s_1}{2}\right)+\sinh \left(\frac{(1-\dim)s_1}{2}\right)\right)}{\dim(\dim-2) d \left(e^{s_1}-1\right)^2 s_1}.
\end{equation*}

Now, we can obtain functions in  dimension two:
\begin{equation*}
\begin{aligned}
H_2(s_1)=&-\frac{e^{\frac{s_1}{2}} \left(e^{s_1} \left(s_1-2\right)+s_1+2\right)}{\left(e^{s_1}-1\right){}^2 s_1},\\
K_2(s_1,s_2)=&\Big(s_1(s_1\hspace{-0.5mm}+\hspace{-0.5mm}s_2)\cosh(s_2)\hspace{-0.5mm}-(s_1\hspace{-0.5mm}-\hspace{-0.5mm}s_2)\left(s_1+s_2+\sinh(s_1)+\sinh(s_2)-\sinh(s_1+s_2)\right)\\
&
\quad -s_2(s_1+s_2) \cosh(s_1)\Big){\rm csch}(\frac{s_1}{2}) {\rm csch}(\frac{s_2}{2}){\rm csch}^2(\frac{s_1+s_2}{2})/(4 s_1 s_2 (s_1+s_2)).
\end{aligned}
\end{equation*}
we have $-4H_2=H$ and $-2K_2=K$ where $K$ and $H$ are the functions found in \cite{Connes-Moscovici2014,Fathizadeh-Khalkhali2013}. 
The difference is coming from the fact that the noncommutative parts of the results in \cite[Section 5.1]{Fathizadeh-Khalkhali2013} are $\lap(\log(e^{h/2}))=\frac12\lap(h)$ and $\Box(\log(e^{h/2}))=\frac14\Box(h)$.

 The functions for dimension four, with the same conformal factor $f(t)=e^t$  and substitution $t_j=\sum_{m=0}^js_m$, gives the following  which up to a negative sign are in complete agreement with the results from \cite{Fathizadeh-Khalkhali2015}:
 \begin{equation*}
\begin{aligned}
H_4(s_1)=\frac{1-e^{s_1}}{2e^{s_1} s_1},\quad 
K_4(s_1,s_2)=\frac{\left(e^{s_1}-1\right) \left(3 e^{s_2}+1\right) s_2-\left(e^{s_1}+3\right) \left(e^{s_2}-1\right) s_1}{4 e^{s_1+s_2}s_1 s_2 \left(s_1+s_2\right)}.
\end{aligned}
\end{equation*}

To recover the functions for curvature of a noncommutative three torus equipped with a conformally flat metric \cite{Khalkhali-Motadelro-Sadeghi2016,Dong-Ghorbanpour-Khalkhali2018}, we need to set  $f(t)=e^{2t}$ and $t_0=s_0,\, t_1=s_0 + s_1/3$ and $t_2=s_0 + (s_1+s_2)/3$. 
Then up to a factor of $e^{-s_0}$, we have
 \begin{equation*}
\begin{aligned}
H_3(s_1)=\frac{4-4e^{\frac{s_1}{3}}}{e^{\frac{s_1}{6}}(s_1 e^{\frac{s_1}{3}}+1) },\quad 
K_3(s_1,s_2)=\frac{6(e^{\frac{s_1}{3}}-1) (3 e^{\frac{s_2}{3}}+1) s_2-6 (e^{\frac{s_1}{3}}+3)(e^{\frac{s_2}{3}}-1) s_1}{ e^{\frac{s_1+s_2}6}(e^{\frac{s_1+s_2}{3}}+1) s_1 s_2 (s_1+s_2)}.
\end{aligned}
\end{equation*}

Finally, we would like to  check the classical limit of our results as $\theta\to 0$.
In the commutative case,  the scalar curvature of a conformally flat metric   $\tilde{g}=e^{2h}g$ on a $\dim$-dimensional space reads 
\begin{equation*} 
 {\tilde {R}}=-2(\dim-1)e^{-2h}g^{jk}\partial_j \partial_k(h) -(\dim-2)(\dim-1)e^{-2h} g^{jk}\partial_j(h)\partial_k(h).
 \end{equation*}
If we let $f(t)=e^{-2t}$, we have the limit
\begin{equation*}
\begin{aligned}
\lim_{t_0,t_1\to t} K_{\dim}(t_0,t_1)&=\frac{1}{3} (d-1) e^{(\dim-2 )t},\quad \lim_{t_0,t_1,t_2\to t}H_{\dim} (t_0,t_1,t_2)&=\frac{1}{6} (\dim-2) (\dim-1) e^{ (\dim-2) t}.
\end{aligned}
\end{equation*}
We should also add that since $\delta_j\to -i \partial_j$ as $\theta\to 0$, we have $\lap(h)\to -g^{jk}\partial_j \partial_k(h) $ and $\Box(h)\to -g^{jk}\partial_j(h)\partial_k(h)$.
Therefore, our result recovers the classical result up to the factor of $\sqrt{|g|}e^{\dim h}/6 $.
The factor $\sqrt{|g|}e^{\dim h}$ represents the volume form in the scalar curvature density and the factor $1/6$ is due to our choice of normalization in \eqref{scalarcurvaturedef}.

\begin{remark}
The rationality of the heat trace densities in different situations \cite{Fathizadeh-Ghorbanpour-Khalkhali2014,Fathizadeh-Marcolli2017} has been studied. 
Here, we have a rationality of the second density for the conformally flat metric on a noncommutative tori and it can be expected to be true in the higher terms of the heat trace densities.
On the other hand, finding relations between the functions $K_d$ and $H_d$ in different dimensions can shed more light on  the structure of these functions.
\end{remark}

\subsection{Twisted product functional metrics}\label{twistedmetricsec}
In this subsection, we shall compute the scalar  curvature density of a noncommutative $\dim$-torus equipped with a class of functional metrics, which  
following the standard convention in the differental geomtry will be called a twisted product metric (see \cite{Kazan-Sahin2013} and references therein).
\begin{definition}
Let $g$ be an $r\times r$ and $\tilde{g}$ be a $(\dim-r)\times (\dim-r)$ positive definite real symmetric  matrices and assume $f$ is a positive function on the real line. 
We call the functional metric 
\begin{equation}\label{twistedmetric}
G=f(t)^{-1}g\oplus \tilde{g},
\end{equation}
the {\it twisted product functional metric} with the twisting element $f(h)^{-1}$.
\end{definition}
Some examples of the twisted product metrics on noncommutative tori were already studied.
The asymmetric two torus whose Dirac oeprator and spectral invariants are studied  in \cite{Dbrowski-Sitarz2015} is a twisted product metric for $r=1$.
The scalar and Ricci curvature of noncommutative three torus of twisted product metrics with $r=2$ are studied in \cite{Dong-Ghorbanpour-Khalkhali2018}.
It is worth mentioning  that conformally flat metrics as well as warped  metrics  are two special cases of twisted product functional metrics.

Let us find the terms to be substituted in the formulas from Proposition \ref{b2proposition}, for the twisted product functional metric \eqref{twistedmetric}. 
First, we note that there are functions  $P_{\bullet}$ and $\tilde{P}_{\bullet}$ such that the functions $P_{\bullet}^{ij}$ of functional parts \eqref{noncommutativelaplacetypesym} of the symbol of $\lap_{0,G}$ can be written as 
\begin{equation}
P^{ij}_{\bullet}=\begin{cases}
g^{ij}P_{\bullet} & 1\leq i,j\leq r \\
\tilde{g}^{ij}\tilde{P}_{\bullet} & r< i,j\leq \dim \\
0 & otherwise.
\end{cases}
\end{equation}
The functions $P_{\bullet}$ are exactly the functions in \eqref{symfunconformal} obtained  for the conformally flat metrics in dimension $r$.
However, the other functions $\tilde{P}_{\bullet}$ are given by
\begin{equation*}
\begin{aligned}
&\tilde{P}_2(t_0)= 1,\\
&\tilde{P}_1(t_0,t_1)
=f(t_0)^{\frac{r}4}\big[t_0,t_1;f^{-\frac{r}4}\big]+f(t_0)^{-\frac{r}4}\big[t_0,t_1;f^{\frac{r}4}\big],\\
&\tilde{P}_{0,1}(t_0,t_1)=f(t_0)^{-\frac{r}4}\big[t_0,t_1;f^{\frac{r}4}\big],\\
&\tilde{P}_{0,2}(t_0,t_1,t_2)= 
f(t_0)^{\frac{r}4}\big[t_0,t_1;f^{-\frac{r}2}\big]\big[t_1,t_2;f^{\frac{r}4}\big]+2f(t_0)^{-\frac{r}4}\big[t_0,t_1,t_2;f^{\frac{r}4}\big].
\end{aligned}
\end{equation*}

On the other hand, while evaluating $T$-functions for metric \eqref{twistedmetric}, we first find that 
\begin{equation*}
\begin{aligned}
P^{-1}(\vec{s})_{k\ell} &=\begin{cases} 
g_{k\ell}\Big(\sum_{j=0}^{n} s_j f(t_j)\Big)^{-1} & 1\leq i,j\leq r \\
\tilde{g}_{k\ell} & r< i,j\leq \dim
\end{cases}\\ 
\sqrt{\det P(\vec{s})}&=\frac1{\sqrt{|g||\tilde{g}|}}\Big(\sum_{j=0}^{n} s_j f(t_j)\Big)^{\frac{r}2}.
\end{aligned}
\end{equation*}
This implies that when exactly $k$ of the entries of $\vec{n}$ are less than equal to $r$, then we have a function $T_{\vec{\alpha}}^k$ such that 
$T_{\vec{n},\vec{\alpha}}=\sqrt{|g||\tilde{g}|} \underset{\vec{n}}{\sum \prod}g_{n_in_{\sigma(i)}} T^k_{\vec{\alpha}}$. 
The function $T^k_{\vec{\alpha}}$ is given by
$$T^k_{\vec{\alpha}}(t_0,\cdots ,t_n)=\frac{1}{2^{|\vec{\alpha}|-2}\vec{\beta}!}\int_{\Sigma_n} \prod s_j^{\alpha_j -1} \Big(\sum_{j=0}^{n} s_j f(t_j)\Big)^{-\frac{r+k}2}.$$
Similar to the conformal case,  the function $T^k_{\vec{\alpha}}$ can be computed as derivatives of the divided difference of explicit functions.
\begin{lemma}\label{Talphaklemma}
Let $\vec{\alpha}$ and $\vec{n}$ be  multi-indices such that exactly $k$ entries of $\vec{n}$ are less than equal to $r$.
Then, for the twisted product metric \eqref{twistedmetric} with $f(t)=t$, the $T$-function $T_{\vec{n},\vec{\alpha}}$can be written as 
$$T_{\vec{n},\vec{\alpha}}(t_0,\cdots ,t_n)=\sqrt{|g||\tilde{g}|} \underset{\vec{n}}{\sum \prod}g_{n_in_{\sigma(i)}} T^k_{\vec{\alpha}}(t_0,\cdots ,t_n).$$ 
where, 
\begin{equation}\label{Talphak}
T^k_{\vec{\alpha}}(t_0,\cdots ,t_n)
=\begin{cases}
\frac{2^{2-|\vec{\alpha}|}}{\Gamma(\vec{\beta}){\prod_{j=1}^{|\vec{\alpha}|-1}} (j-\frac{r+k}2)} \partial_t^{\vec{\beta}} \big[t_0,\cdots,t_n;u^{-\frac{r+k}2+|\vec{\alpha}|-1}\big] &|\vec{\alpha}|\neq  \frac{r+k}{2}+1,\\ &\\
\frac{(-2)^{2-|\vec{\alpha}|}}{\Gamma(\vec{\beta})\Gamma(\frac{r+k}2)}\partial_t^{\vec{\beta}}\big[t_0,\cdots,t_n;\log(u)\big] & |\vec{\alpha}|= \frac{r+k}{2}+1. 
\end{cases}
\end{equation}
\qed
\end{lemma}
It can readily  be seen that for $k=2|\vec{\alpha}|-4$ the function $T^{k}_{\vec{\alpha}}$ is equal to $T_{\vec{\alpha}}$ found in \eqref{Tfunctionsconfomallyflat} for the conformally flat case.

\begin{theorem}\label{scalartwisted}
The scalar curvature density of the $\dim$-dimensional noncommutative tori $\nctorus[\dim]$ equipped with the twisted product functional metric \eqref{twistedmetric} with the twisting element $f(h)^{-1}$   is given by
\begin{equation*}
\begin{aligned}
R=\sqrt{|g||\tilde{g}|} \Big(& K_{r}(h_{(0)},h_{(1)})(\lap(h))+H_{r}(h_{(0)},h_{(1)},h_{(2)})(\Box(h))\\
&+\tilde{K}_{r}(h_{(0)},h_{(1)})(\tilde{\lap}(h))+\tilde{H}_{r}(h_{(0)},h_{(1)},h_{(2)})(\tilde{\Box}(h))\Big),
\end{aligned}
\end{equation*} 
where $\tilde{\lap}(h)=\sum_{r<i,j}\tilde{g}^{ij}\delta_i\delta_j(h)$ and $\tilde{\Box}(h)=\sum_{r<i,j} \tilde{g}^{ij}\delta_i(h)\delta_j(h)$ and  $\lap$, $\Box$, $K_r$ and $H_r$ are  given by Theorem \ref{scalarcrvatureconformal}. 
The functions $\tilde{K}_{r}$ and $\tilde{H}_{r}$ for $r\neq 2, 4$ are  given by
$$\tilde{K}_{r}(t_0,t_1)=\tilde{K}_{r}^t(f(t_0),f(t_1))\big[t_0,t_1;f\big],$$
$$\tilde{H}_{r}(t_0,t_1,t_2)=\tilde{H}_{r}^t(f(t_0),f(t_1),f(t_2))\big[t_0,t_1;f\big]\big[t_1,t_2;f\big]+2\tilde{K}_{r}^t(f(t_0),f(t_2))\big[t_0,t_1,t_2;f\big].$$
The functions $\tilde{K}_{r}^t$ and $\tilde{H}_{r}^t$ are 
\begin{align*}
\tilde{K}_{r}^t(x,y)=\frac{(2r-4)(x^2-y^2) x^{\frac{r}2} y^{\frac{r}2}+4 x^2 y^r-4 x^r y^2 }{(r-4) (r-2)x^{\frac34 r} y^{\frac34 r} (x-y)^3},
\end{align*}
and
\begin{align*}
&\tilde{H}_{r}^t(x,y,z)=\frac{2 x^{-\frac{3 r}{4}} y^{-r} z^{-\frac{3 r}{4}}}{(r-4) (r-2) (x-y)^2 (x-z)^3 (y-z)^2}\times\\
&\Big(x^r y^r z^2 (x-y) \left(x^2+2 x (y-2 z)-4 y^2+6 y z-z^2\right)\\
&\,\, +x^r y^{\frac{r}2} z^{\frac{r}2} (x-z)^2 (y-z) \big((r-3) y z-x ((r-3) z+y)\big)\\
&\,\,- x^r y^2 z^r (x-z)^3+ x^{\frac{r}2} y^{\frac{3 r}{2}}z^2 (x-y) (x-z)^2\\
&\,\,-x^{\frac{r}2} y^r z^{\frac{r}2} (x\hspace{-0.55mm}-\hspace{-0.55mm}y) (x\hspace{-0.55mm}-\hspace{-0.55mm}z) (y\hspace{-0.55mm}-\hspace{-0.55mm}z) \big(\hspace{-0.75mm}(r\hspace{-0.55mm}-\hspace{-0.55mm}3) x^2\hspace{-0.55mm}-2 x ((r\hspace{-0.55mm}-\hspace{-0.55mm}2) y+(1\hspace{-0.55mm}-r)z)+z ((r\hspace{-0.55mm}-3) z\hspace{-0.55mm}-2 (r\hspace{-0.55mm}-2) y)\hspace{-0.75mm}\big)\\
&\,\,+x^{\frac{r}2} y^{\frac{r}2} z^r (y-x) (x-z)^2 (y z-(r-3) x (y-z))\\
&\,\,+x^2 y^{\frac{3 r}{2}} z^{\frac{r}2} (x-z)^2 (y-z)
-x^2 y^r z^r (y-z) \left(x^2-6 x y+4 x z+4 y^2-2 y z-z^2\right)\Big).
\end{align*}
\end{theorem}
\begin{proof}
We first separate the terms in the formulas from Proposition \ref{b2proposition} in two parts: 
the part which contains no tilde term and the rest of the terms.
For example the formula for $B^{ij}_ {21}$ can be rewritten as 
\begin{equation*}
\begin{aligned}
&B^{ij}_ {21}(t_ 0, t_ 1)=\sqrt{| g||\tilde {g}|}\times \\
&\Big(g^{ij}\left(-T^0_{1, 1}(t_0,t_1)P_{0,1}(t_0,t_1)
+ T^2_{2,1}(t_0,t_1)t_0(2P_1(t_0,t_1)+r)  
-4(r+2)T^4_{3,1}(t_0,t_1) t_0^2\right)\\
&\quad+\tilde{g}^{ij}\big(-T^0_{1,1}(t_0,t_1)\tilde{P}_{0,1}(t_0,t_1) 
+2T^0_{2,1}(t_0,t_1)\tilde{P}_1(t_0,t_1)
+r T^2_{2,1}(t_0,t_1)
-4rT^2_{3, 1}(t_0,t_1)\big)\Big).
\end{aligned}
\end{equation*}
The second part, multiples of $\tilde{g}^{ij}$,  will  be denoted by $\tilde{K}_r^t$, while the first part, multiple of $g^{ij}$, is nothing but  $K_r^t(t_0,t_1)$; the function found for the conformally flat case in Theorem \ref{scalarcrvatureconformal}.
In other words, we have
\begin{equation*}
B_{2,1}^{ij}(h_{(0)},h_{(1)})(\delta_i\delta_j(h))=\sqrt{|g||\tilde{g}|}\Big(  K^t_{r}(h_{(0)},h_{(1)})(\lap(h))+\tilde{K}^t_{r}(h_{(0)},h_{(1)})(\tilde{\lap}(h))\Big),
\end{equation*}
where $\lap(h)=\sum_{i,j\leq r} g^{ij}\delta_i\delta_j(h)$ and $\tilde{\lap}(h)=\sum_{i,j> r}\tilde{g}^{ij}\delta_i\delta_j(h)$.
Similarly, this can be done for $B_{22}^{ij}$, i.e. there are  functions   $\tilde{H}_r^t$ and $H_r^t$ such that
\begin{equation*}
\begin{aligned}
&B_{2,2}^{ij}(h_{(0)},h_{(1)},h_{(2)})(\delta_i(h)\cdot\delta_j(h))=\\
&\sqrt{|g||\tilde{g}|}\Big(  H_{r}(h_{(0)},h_{(1)},h_{(2)})(\Box(h))+\tilde{H}_{r}(h_{(0)},h_{(1)},h_{(2)})(\tilde{\Box}(h))\Big).
\end{aligned}
\end{equation*}
Here we use the notation  $\Box(h)=\sum_{i,j\leq r} g^{ij}\delta_i(h)\cdot\delta_j(h)$ and  $\tilde{\Box}(h)=\sum_{i,j> r}\tilde{g}^{ij}\delta_i(h)\cdot\delta_j(h)$.
The rest is just substitution of  each term into the formula for $B_{21}^{ij}$ and $B_{22}^{ij}$.
\end{proof}
When the selfadjoint element $h\in \Ai[\dim]$ has the property that  $\delta_j(h)=0$ for $1\leq j\leq r$, 
we call the twisted product functional metric \eqref{twistedmetric} a {\it warped functional metric} with the warping element $1/f(h)$. 
\begin{corollary}
The scalar curvature density of a warped product of $\tilde{g}$ and $g$ with the warping element $1/f(h)$  is given by
\begin{equation*}
R=\sqrt{|g||\tilde{g}|} \Big(\tilde{K}_{r}(h_{(0)},h_{(1)})(\tilde{\lap}(h))+\tilde{H}_{r}(h_{(0)},h_{(1)},h_{(2)})(\tilde{\Box}(h))\Big).
\end{equation*} 
\end{corollary}
\begin{proof}
It is enough to see that  $\lap(h)$ and $\Box(h)$ vanish for the warped metric.
\end{proof}

For $r=2$ and $r=4$, functions $\tilde{H}_r$ and $\tilde{K}_r$ are the limit of the functions given in Theorem \ref{scalartwisted} as $r$ approache $2$ or $4$. 
This is because of the fact that for these values of $r$, some of $T_{\vec{\alpha}}^k$ functions are the  limit case of formulas found in \eqref{Talphak}. 
The functions for $r=2$ are given by
\begin{equation*}
\tilde{K}_2(x,y)=\frac{-x^2+y^2+2 x y \log (\frac{x}{y})}{\sqrt{xy} (x-y)^3},
\end{equation*}
and  
\begin{equation*}
\begin{aligned}
\tilde{H}_2(x,y,z)=&\frac{1}{2 y\sqrt{xz}  (x-y)^2 (x-z)^3 (y-z)^2}\Big(-y (x+y) (x-z)^3 (y+z) \log (y)\\
&-z (x-y)^2  \left(-3 x^2 y+x^2 z-8 x y^2+10 x y z-2 x z^2+y z^2+z^3\right)\log (z)\\
&+x  (y-z)^2 \left(x^3+x^2 y-2 x^2 z+10 x y z+x z^2-8 y^2 z-3 y z^2\right) \log (x)\\
&+ 2 y (y-x) (x-z) (x+z) (z-y) (x-2 y+z)\Big).
\end{aligned}
\end{equation*}
For $r=4$, on the other hand, we have
\begin{equation*}
\tilde{K}_4(x,y)=\frac{x^2-y^2-(x^2+y^2) \log (\frac{x}{y})}{x y (x-y)^3},
\end{equation*}  
and
\begin{align*}
&\tilde{H}_4(x,y,z)=\frac1{2 x (x - y)^2 y^2 (x - z)^3 (y - z)^2 z}\times\\
& \Big(\left(x^2+y^2\right) (x-z)^3 \left(y^2+z^2\right) \log (y)\\
&\quad\,\, +\log (x) (y-z)^2 \left(x^4 y+x^4 z-6 x^3 y^2-2 x^3 y z-2 x^3 z^2+3 x^2 y^3+x^2 y^2 z+x^2 y z^2\right.\\
&\quad\,\, \qquad \qquad\qquad\qquad\left.+x^2 z^3+2 x y^3 z -4 x y^2 z^2+3 y^3 z^2+y^2 z^3\right)\\
&\quad\,\, -\log (z) (x-y)^2\left(x^3 y^2+x^3 z^2+3 x^2 y^3-4 x^2 y^2 z+x^2 y z^2-2 x^2 z^3+2 x y^3 z+x y^2 z^2\right.\\
&\quad\,\, \qquad \qquad\qquad\qquad\left.-2 x y z^3+x z^4+3 y^3 z^2-6 y^2 z^3+y z^4\right)\\
&\quad\,\, -2 (x-y) (x-z) (y-z) \left(x^3 z+x^2 y^2-2 x^2 z^2-2 x y^3+2 x y^2 z+x z^3-2 y^3 z+y^2 z^2\right)\Big).
\end{align*}

In \cite[section 4.1]{Dong-Ghorbanpour-Khalkhali2018}, the scalar curvature density of twisted product functional metric on noncommutative three torus for $f(t)=e^{2t}$ and $r=2$ is found. 
This result can be recovered from our formulas given in Theorem \ref{scalartwisted} by setting $t_0=s_0$, $t_1=s_0+s_1/2$ and $t_2=s_0+s_1/2+s_2/2$.

We would like to conclude this section by a remark on multiply twisted metrics and commuting families of metrics.
Let $[g^{ij}(t)]$ be a commuting family of positive definite real symmetric matrices.
Such a family diagonalizes simultaneously; i.e. there exist an orthogonal matrix $O$ such that 
$$[g^{ij}(t)]=O^{-1} {\rm diag}(\lambda_1(t),\cdots,\lambda_{\dim}(t))O.$$
Such a family is included in the class of so called  {\it multiply twisted product metrics}.
Let $f_j$'s be positive functions and take $g_j$'s to be $n_j$ dimensional constant metrics.
Then
$$G=\bigoplus_j f_j(h)g_j$$ 
defines a functional metric on $\nctorus[\dim]$, where $\dim=\sum n_j$.
The $T$-functions computed in Example \ref{doubalytsitedtf} are for metrics on four tori with two twisting factors $f$ and $\tilde{f}$.
Very similar computations would work for multiply twisted product of even dimensional tori equipped with constant metrics.
We shall use the $T$-functions found for the doubly twisted metric on $\nctorus[4]$ to compute its total scalar curvature in the next section.

\section{Total scalar curvature of functional metrics}\label{Totalcurvaturesection}
The goal of this section is to show how the trace of the second density of the  heat trace of a  Laplace type $h$-differential operator $P$ on $\nctorus[\dim]$ can be simplified. 
When $P=\lap_{0,g}$ for some functional metric $g(h)$, we shall call 
$$\varphi(R)=(4\pi)^{\frac{d}2}\varphi(b_2(\lap_{0,g}))$$ 
the {\it total scalar curvature} of the metric $g(h)$ on the noncommutative torus.
In this section, we show that the total scalar curvature  for all functional metrics on $\nctorus$  vanishes.
This generalizes the Gauss-Bonnet theorem proved in \cite{Connes-Tretkoff2011,Fathizadeh-Khalkhali2012} for the conformally flat metrics on $\nctorus$.     

\subsection{Trace of the second density}
We start with the following lemma which essentially is an avatar of  the tracial property $\phi(ab)=\phi(ba)$ when the trace acts on an element  written in the contraction form. 
Versions of this can be found in \cite{Liu2018-II}.
\begin{lemma}\label{traceofcontraction} 
Let $h$ be a selfadjoint element of a C$^\ast$-algebra $A$ and let $\phi$ be a trace on $A$.
For every smooth function $F:\mathbb{R}^{n+1}\to \mathbb{C}$, we have
\begin{itemize}
\item[i)] $\phi\left(F(h_{(0)},\cdots, h_{(n)})\big(b_1\cdot\cdots  b_n\big)\right)= 
\phi\left(F(h_{(0)},\cdots,h_{(n-1)},h_{(0)})\big(b_1\cdot \cdots  b_{(n-1)}\big) b_n\right),$
\item[ii)] $\phi\left(F(h_{(0)},\cdots,h_{(n)})\big(b_1\cdots b_n\big)b_{n+1}\right)
= \phi\left(b_1 F(h_{(1)},\cdots,h_{(n)},h_{(0)})\big(b_2\cdots  b_{n+1}\big)\right).$
\end{itemize}
\end{lemma}
\begin{proof}
To prove the identities, first we write the element in the contraction form as an integral given in Remark \ref{integralform}.
After applying the trace property to the integrand, we can then use the rearrangement lemma \ref{rearrangmentlemma} to put the element back into the contraction form.
\end{proof}
For a Laplace type $h$-differential operator $P$ on $\nctorus[\dim]$, this lemma implies that  
\begin{equation}\label{1stepsimbtilde}
\varphi(b_2(P))=(4\pi)^{-\dim/2}\varphi\left(\tilde{B}_{2,1}^{ij}(h) \delta_i\delta_j(h)+\tilde{B}_{2,2}^{ij}(h_{(0)},h_{(1)})\big(\delta_i(h)\big)\delta_j(h)\right),
\end{equation}
where 
\begin{equation*}
\tilde{B}_{2,1}^{ij}(t)=B^{ij}_{2,1}(t,t),\quad \tilde{B}_{2,2}^{ij}(t_0,t_1)=B_{22}(t_0,t_1,t_0).
\end{equation*}

We can hence compute and simplify the functions $\tilde{B}_{2,1}^{ij}(t)$ and $\tilde{B}_{2,2}^{ij}(t_0,t_1)$ for a    Laplace type $h$-differential operator $P$:
\begin{align*}
\tilde{B}_{21}(t)=\frac{1}{\sqrt{\det P_2(t)}}\left(-P_{0,1}^{ij}(t,t)+\frac{1}{2} P_1^{ji}(t,t)+\frac1{12}\tr\left(P^{-1}(t)P'_2(t)\right) \, P_2^{ij}(t) \,-\frac{1}{3} \,P'{}_2^{ij}(t)\right),
\end{align*}
and 
\begin{align*}
\tilde{B}_{2,2}(t_0,t_1)=
&\frac1{\sqrt{\det P_2(t_0)}} \left( -P_{0,2}^{ij}(t_0,t_1,t_0)+
\big[t_0,t_1;P_1^{ij}(t_0,\cdot)\big] -\frac{2}{3} P'{}_2^{ij}(t_0) \frac{1}{t_0-t_1}\right.\\
&\qquad\qquad \qquad \,\left.+P_2^{ij}(t_0)\,\frac1{(t_0-t_1)^2} +\frac{ \tr(P_2(t_0)^{-1}P'{}_2(t_0))}{6(t_0-t_1)}P_2^{ij}(t_0)\right)\\
&+T_{k\ell;2,1}(t_0,t_1)
\left(
P_1^{ik}(t_0,t_1)
-2P_2^{ik}(t_0)\frac1{t_0-t_1}\right) 
\left(P^{j\ell}_1(t_1,t_0)  
-2 \big[t_0,t_1;P_2^{j\ell}\big]\right)\\
&-2T_{k\ell ;1,2}(t_0,t_1) \left(P_1^{ik}(t_0,t_1)
-2P_2^{ik}(t_0)\frac1{t_0-t_1}\right)P_2^{j\ell}(t_1)\frac1{t_0-t_1}\\
&+T_{;1,1}(t_0,t_1) \left( P_1^{ij}(t_0,t_1)
- 2 P_2^{ij}(t_0)\,\frac1{t_0-t_1}\right)\frac1{t_0-t_1}.
\end{align*}
To  write these functions in the present form, we also used Lemma \ref{Tfunctionsrelations} and some of the basic properties of $T$-functions from Remark \ref{basicpropTfunctions}. 
To illustrate these steps, we show the complete work required to be done on the tenth term of $B_{22}^{ij}$ from Proposition \ref{b2proposition}.
\begin{align*}
\lim_{t_2\to t_0}&- 4T_{k\ell mn;2,1,1}(t_0,t_1,t_2)\, P_2^{ik}(t_0)\,\big[t_0,t_1;P_2^{ j\ell}\big]\,\big[t_1,t_2;P_2^{mn}\big]\,\\
&=- 4T_{k\ell mn;2,1,1}(t_0,t_1,t_0)\, P_2^{ik}(t_0)\,\big[t_0,t_1;P_2^{ j\ell}\big]\,\big[t_1,t_0;P_2^{mn}\big]\,\\
&=- 4T_{k\ell mn;3,1}(t_0,t_1)\big[t_1,t_0;P_2^{mn}\big]\, P_2^{ik}(t_0)\,\big[t_0,t_1;P_2^{ j\ell}\big]\,\,\\
&=- 4(T_{k\ell;2,1}(t_0,t_1)-T_{k\ell;3}(t_0))\, P_2^{ik}(t_0)\,\big[t_0,t_1;P_2^{ j\ell}\big]\frac1{t_0-t_1}\,\\
&=- 4T_{k\ell;2,1}(t_0,t_1) P_2^{ik}(t_0)\,\big[t_0,t_1;P_2^{ j\ell}\big]\frac1{t_0-t_1}\,\\
&\quad+ 4\frac{P_2^{-1}(t_0)_{k\ell}}{4\sqrt{\det P_2(t_0)}}\, P_2^{ik}(t_0)\,(P_2(t_0)^{j\ell}-P_2(t_1)^{j\ell})\frac1{(t_0-t_1)^2}\,\\
&=- 4T_{k\ell;2,1}(t_0,t_1) P_2^{ik}(t_0)\,(P_2^{ j\ell}(t_0)-P_2^{ j\ell}(t_1))\frac1{(t_0-t_1)^2}\,\\
&\quad+\frac{P_2^{-1}(t_0)_{k\ell}}{\sqrt{\det P_2(t_0)}}\, P_2^{ik}(t_0)P_2(t_0)^{j\ell}\frac1{(t_0-t_1)^2}\,
- \frac{P_2^{-1}(t_0)_{k\ell}}{\sqrt{\det P_2(t_0)}}\, P_2^{ik}(t_0)P_2(t_1)^{j\ell}\frac1{(t_0-t_1)^2}\,\\
&=- 4T_{k\ell;2,1}(t_0,t_1) P_2^{ik}(t_0)\,P_2^{ j\ell}(t_0)\frac1{(t_0-t_1)^2}\,
+ 4T_{k\ell;2,1}(t_0,t_1) P_2^{ik}(t_0)\,P_2^{ j\ell}(t_1)\frac1{(t_0-t_1)^2}\,\\
&\quad+ \frac{1}{\sqrt{\det P_2(t_0)}}P_2(t_0)^{ji}\frac1{(t_0-t_1)^2}\,
- \frac{1}{\sqrt{\det P_2(t_0)}}\,P_2(t_1)^{ji}\frac1{(t_0-t_1)^2}.
\end{align*}

Further, we use compatibility of the trace  and  the derivations, $\varphi \circ \delta_j=0$, and the Leibniz  rule to turn the term $\tilde{B}_{2,1}^{ij}(h)\delta_i\delta_j(h)$ into the form of the other summand, 
\begin{equation*}
\begin{aligned}
 \varphi\left(\tilde{B}_{2,1}^{ij}(h)\delta_i\delta_j(h)\right)
 &= \varphi\left(\delta_i(\tilde{B}_{2,1}^{ij}(h)\delta_j(h))\right)-\varphi\left(\delta_i(\tilde{B}_{2,1}^{ij}(h))\delta_j(h)\right)\\
 &=-\varphi\left(\delta_i(\tilde{B}_{2,1}^{ij}(h))\delta_j(h)\right)\\ 
 &=-\varphi\left([h_{(0)},h_{(1)};\tilde{B}_{2,1}^{ij}]\big(\delta_i(h)\big)\delta_j(h)\right).
\end{aligned}
\end{equation*}
Then, we put both of the terms together and rewrite \eqref{1stepsimbtilde} as
\begin{equation*}
\varphi(b_2(P))=(4\pi)^{-\dim/2}\varphi\left(F^{ij}(h_{(0)},h_{(1)})\big(\delta_i(h)\big)\delta_j(h)\right),
\end{equation*}
where $F^{ij}(t_0,t_1)=[t_0,t_1;\tilde{B}_{2,1}^{ij}]+\tilde{B}_{2,2}^{ij}(t_0,t_1)$.
Applying identity $ii)$ of Lemma \ref{traceofcontraction}, we find that  
\begin{equation*}
\varphi\left((F^{ij}(h_{(0)},h_{(1)})-F^{ji}(h_{(1)},h_{(0)}))\big(\delta_i(h)\big)\delta_j(h)\right)=0.
\end{equation*} 
Therefore, only the expression $\frac12(F^{ij}(t_0,t_1)+F^{ji}(t_1,t_0)))$, denoted by $F_S^{ij}(t_0,t_1)$, may contribute to the value of $\varphi(b_2(P))$.  
This fact was first noticed in \cite{Connes-Tretkoff2011} for the conformally flat metrics and used in other works \cite{Fathizadeh-Khalkhali2013,Dbrowski-Sitarz2015} to prove a Gauss-Bonnet theorem in dimension two.

In the case of conformally flat metrics, since we have $F^{ij}=F^{ji}$, the function $F_S^{ij}$ is indeed the symmetric part of the function $F^{ij}$ and this is the rationale behind our notation $F_S^{ij}$.

\begin{example}\label{totalcuroftwistedexample}
In this example, we examine what we have discussed so far for the total scalar curvature of twisted product of flat $(\nctorus[r],g)$ and $(\nctorus[\dim-r],\tilde{g})$ with twisting factor $f(t)^{-1}$ (see the Section \ref{twistedmetricsec}). 
We assume that $f(t)=t$ and from Theorem \ref{scalartwisted} we have
\begin{equation*}
\tilde{B}^{ij}_{2,1}(x)=
\begin{cases}
-\frac{1}{6} (r-1) x^{-\frac{r}{2}} &1\leq i,j\leq r\\
-\frac{1}{6} r x^{-\frac{r}{2}-1} & r< i,j\leq d\\
0 & otherwise.
\end{cases}
\end{equation*} 
The $\tilde{B}^{ij}_{22}(x,y)$ for  $1\leq i,j\leq r$ is given by
\begin{align*}
\Big(&
-6  x^{\frac{3 r}{2}} y^3
+3 x^r y^{\frac{r}{2}+1} ((r-2) y-r x) ((r-1) y-(r-2) x)\\
&
+x^{\frac{r}2} y^r \left(-r (r-1) (r-2) (x-y)^3+6 x y ((r-1) x-(r-2) y)\right)\\
&
+3 x^2 y^{\frac{3 r}{2}} (rx-(r-2)y )\Big)/\Big({3 (r-2) r (x-y)^4}{x^{r} y^{r}}\Big).
\end{align*}
Also, $\tilde{B}^{ij}_{22}(x,y)$ for  $r< i,j\leq d$ is given by
\begin{align*}
\Big(&
-6 x^{\frac{3 r}{2}+1} y^2 +3 x^{r+1} y^{r/2} \left((r-3) (r-2) x^2-(r-4) (2 r-3) x y+(r^2-6r+4) y^2\right)\\
&
+x^{\frac{r}2} y^r \big(r(r-4)(r-2) y^3-3 (r^3-6r^r+8r+2) x y^2\\
&\qquad \qquad\qquad +3 (r^3-6r^2+6r+8) x^2 y-(r-2) (r^2-4r-6) x^3\big)\\
&
+3 x^2 y^{\frac{3 r}{2}} ((r-2) x-(r-4) y)\Big)/\Big({3 (r-4) (r-2) (x-y)^4}{x^{r+1} y^{r}}\Big).
\end{align*}
Therefore, there are two symmetric functions $F_S$ and $\tilde{F}_S$ such that 
\begin{align*}
\varphi(R)=
&\sqrt{|g||\tilde{g}|}\varphi\left(g^{ij}F_S(f(h_{(0)},h_{(1)}))[h_{(0)},h_{(1)};f](\delta_i(h))\delta_j(h)\right)\\
&+\sqrt{|g||\tilde{g}|}\varphi\left(\tilde{g}^{ij}\tilde{F}_S(f(h_{(0)},h_{(1)}))[h_{(0)},h_{(1)};f](\delta_i(h))\delta_j(h)\right).
\end{align*}
These functions are given by
\begin{align*}
F_S(x,y)&=\frac{x^{-r} y^{-r}}{2 r (x-y)^3} \left(y x^{\frac{r}2}+x y^{\frac{r}2}\right) \left((r-1) x^{\frac{r}2} y^{\frac{r}2} (x-y)-y x^r+x y^r\right),
\\
\tilde{F}_S(x,y)&=\frac{x^{-r} y^{-r}}{2 (r-2) (x-y)^3} \left(x^{\frac{r}2} y^r (rx +(1-r)y)-y x^{\frac{3 r}{2}}+x^r y^{\frac{r}2} ((r-1) x-r y)+x y^{\frac{3 r}{2}}\right).
\end{align*}
\end{example}

Applied to  the Laplacian of a functional metric $g$ on $\nctorus[\dim]$, what we have covered so far in this section  gives a simplified term whose trace is equal to the  total scalar curvature of $(\nctorus[\dim], g)$, which we summarize in the following proposition.
\begin{proposition}\label{totalcurvature}
The total scalar curvature of $\nctorus[\dim]$ equipped with a functional metric $g$ is given by 
\begin{equation*}
\varphi(R)=\varphi\left(F_S^{ij}(h_{(0)},h_{(1)})\big(\delta_i(h)\big)\delta_j(h)\right),
\end{equation*}
where $F_S^{ij}(t_0,t_1)$ is given by
\begin{equation*}
\begin{aligned}
F_S^{ij}(t_0,t_1)= \frac1{2(t_0-t_1)^2}\Big(
&A^{ij}\sqrt[4]{|g|(t_0)|g|(t_1)}-2A^{ij} T_{;1,1}(t_0,t_1)\\
&+T_{k,l;1,2}(t_0,t_1) \big(2  A^{ik} g^{lj}(t_1)+2A^{kj}  g^{il}(t_1)-A^{ik}A^{lj}\big)\\
&+T_{k,l;2,1}(t_0,t_1)\big (2A^{ik}g^{lj}(t_0)+2  A^{kj} g^{il}(t_0) -A^{ik}A^{lj}\big)\Big),
\end{aligned}
\end{equation*} 
and 
$A^{ij}=
{|g|^{\frac14}(t_0 )}{|g|^{\frac{-1}4} (t_1 )}g^{i j} (t_0 )
+{|g|^{\frac14} (t_1 )}{|g|^{\frac{-1}4}(t_0)} g^{ij} (t_1 ).$\qed
\end{proposition}

\begin{example}\label{dtflatT4}
We conclude this section by computing the functions $F^{ij}_S$ for the doubly twisted product functional metric $f(h)^{-1}g\oplus \tilde{f}(h)^{-1}\tilde{g}$ on $\nctorus[4]$ whose $T$-functions was found in Example \ref{doubalytsitedtf}. 
Using Proposition \ref{totalcurvature}, we have in this case
\begin{align*}
 A=(\tilde{f}(t_0)+\tilde{f}(t_1)) \sqrt{\frac{f(t_0) f(t_1)}{\tilde{f}(t_0) \tilde{f}(t_1)}} g^{-1}\oplus
 (f(t_0)+f(t_1)) \sqrt{\frac{\tilde{f}(t_0) \tilde{f}(t_1)}{f(t_0) f(t_1)}} \tilde{g}^{-1}.
\end{align*}
Then by substituting all the terms in Proposition \ref{totalcurvature},   we find that  for $1\leq i,j\leq 2 $
\begin{align*}
&F_S^{ij}(t_0,t_1)=\frac{\sqrt{|g| |\tilde{g}|}(\tilde{f}(t_0)^2-\tilde{f}(t_1)^2)g^{ij}}{4(t_0-t_1)^2(f(t_1) \tilde{f}(t_0)-f(t_0) \tilde{f}(t_1))^2}\times \\
&\hspace*{-0.2cm}\left(\frac{f (t_1 )}{\tilde{f} (t_1 )}  \Big(f (t_0 )  \big(\log  (\frac{f (t_0 )  \tilde{f} (t_1 )}{f (t_1 )  \tilde{f} (t_0 )} )+1 \big)+f (t_1 )\Big)
\hspace*{-0.1cm}-\hspace*{-0.1cm}\frac{f (t_0 )}{\tilde{f} (t_0)}  \Big(f (t_1 )  \big(\hspace*{-0.1cm} \log  (\frac{f (t_1 )  \tilde{f} (t_0 )}{f (t_0 )  \tilde{f} (t_1 )} )+1 \big)+f (t_0 ) \Big)\hspace*{-0.1cm}\right).
\end{align*} 
The formula for functions $F_S^{ij}(t_0,t_1)$ for $3\leq i,j\leq 4 $, can be obtained   by interchanging $f$ with $\tilde{f}$ and  $g^{-1}$ with $\tilde{g}^{-1}$ in the above formula.
\end{example}

\subsection{Dimension two and a Gauss-Bonnet theorem}\label{GBdim2sec}
In this subsection, we study the functions $F_S^{ij}$ in dimension two. 
We show that these functions vanish  for the noncommutative two torus equipped with a functional metric $g$.
This means that the total scalar curvature of  $(\nctorus,g)$ is independent of $g$. 
Similar to the conformally flat metrics \cite{Connes-Tretkoff2011,Fathizadeh-Khalkhali2012}, we call this result the Gauss-Bonnet theorem for functional metrics.
\begin{theorem}[Gauss-Bonnet Theorem]\label{GBindim2}
The total scalar curvature $\varphi(R)$ of the noncommutative two tori equipped with a functional  metric  vanishes, hence it is independent of the metric.
\end{theorem}
Before we give the proof of the theorem, we make some comments and prove a lemma.
To prove the theorem, we evaluate the functions $F_S^{ij}$ for dimension two.
What makes it possible for us to do this is that $T$-functions present in the formula for the total scalar curvature from Proposition \ref{totalcurvature} can be explicitly evaluated in dimension two for all functional metrics.
In particular, the function $T_{;1,1}$ in dimension two is given by a basic integral; that is the integral of an inverse square root of a quadratic function, and we have 
\begin{equation*}
T_{;1,1}(t_0,t_1)=
\begin{cases} 
\frac{1}{\sqrt{a}}\ln\left(\frac{2a+b+2\sqrt{a^2+ab+ac}}{b+2\sqrt{ac}}\right)& a>0\\
\frac{1}{\sqrt{-a}}\sin^{-1}\left(\frac{-2a+b}{\sqrt{b^2-4ac}}\right)-\frac{1}{\sqrt{-a}}\sin^{-1}\left(\frac{b}{\sqrt{b^2-4ac}}\right)
& a<0
\end{cases}
\end{equation*}
where $c=\det(P_2(t_1))$, $b=-\det(P_2(t_1))\tr(I-P_2(t_0)P_2(t_1)^{-1})$ and $a=\det(P_2(t_1)-P_2(t_0)).$

The form of $T_{;1,1}(t_0,t_1)$  depends on the sign of $a$.
In the conformal case, $a=(f(t_1)-f(t_0))^2$ which is always positive.
In  general, however, there are functional metrics for which $a$ is negative. 
For instance, consider the diagonal functional metric with the diagonal entries $t$ and a positive decreasing function $f(t)$. 
For all nonzero $t$, then $\det(P_2(t_1)-P_2(t_0))<0$.
On the other hand, the case $a=0$, which is the limit of the other cases, can also happen.
The diagonal functional metric with one constant diagonal entry is an instance of such a case. 
While all three cases are possible at the same time,  we will deal with them separately.
This is simply because the conditions $a>0$ or $a<0$ are open conditions and we can always choose to work with a selfadjoint element $h$  whose spectrum is either in  $a^{-1}(0,+\infty)>0$ or $a^{-1}(-\infty,0)<0$.

By the property $T_{k,l;2,1}(t_0,t_1)=T_{k,l;1,2}(t_1,t_0)$, we only need to find $T_{k,l;2,1}$. 
This is done in the following lemma in which we write  $T_{k,l;2,1}$ in terms of $T_{1,1}$.
\begin{lemma}\label{Tfunctionsdimtwo}
In  dimension two, when $\det(P_2(t_0)-P_2(t_1))\neq 0$, we have
\begin{equation*}
\begin{aligned}
&T_{k,l;2,1}(t_0,t_1)=\frac1{2|g(t_0)-g(t_1)|} T_{1,1}( t_0 , t_1 ) \left(\frac{g_{k l}( t_0 )}{|g|( t_0 )}-\frac{g_{k l}( t_1 )}{|g|( t_1 )}\right)\\
&+\frac{|g|( t_0 )^{3/2} |g|( t_1 ) (\sqrt{|g|( t_0 )} \sqrt{|g|( t_1 )}+|g|( t_1 )-|g(t_0)-g(t_1)|)}{|g(t_0)-g(t_1)| \left(|g(t_0)-g(t_1)|-(\sqrt{|g|( t_0 )}+\sqrt{|g|( t_1 )})^2\right)}\left(\frac{g_{k l}( t_0 )}{|g|( t_0 )}-\frac{g_{k l}( t_1 )}{|g|( t_1 )}\right)\\
&+\frac{|g|( t_0 )^{3/2}}{\left(\sqrt{|g|( t_0 )}+\sqrt{|g|( t_1 )}\right)^2-|g(t_0)-g(t_1)|}g_{kl}( t_1 ).
\end{aligned}
\end{equation*}
Moreover, for $\det(P_2(t_0)-P_2(t_1))= 0$, we have
\begin{equation*}
\begin{aligned}
T_{k,l;2,1}(t_0,t_1)=&\frac{|g|( t_0 )^{3/2} }{\left(\sqrt{|g|( t_0 )}+\sqrt{|g|( t_1 )}\right)^2}g_{kl}( t_1 ),\qquad 
T_{;1,1}(t_0,t_1)=&\frac{\sqrt{|g|( t_0 )|g|( t_1 )}}{\sqrt{|g|( t_0 )}+\sqrt{|g|( t_1 )}}.
\end{aligned}
\end{equation*}
\end{lemma} 
\begin{proof}
We substitute the formula of $T$-functions from Lemma \ref{Tfunctionsdimtwo}  in the formula for the total scalar curvature $\varphi(R)$ found in Proposition \ref{totalcurvature}. 
We find that $F_S^{ij}$ is the $(i,j)$ component of the matrix 
\begin{equation*}
\begin{aligned}
&\left(2 |g(t_0)-g(t_1)| |g|^{\frac14}(t_0)|g|^{\frac14}(t_1) (t_0-t_1)^2 \big(|g(t_0)-g(t_1)|-(|g|^{\frac12}(t_0)+|g|^{\frac12}(t_1))^2\big)\right)^{-1}\\
&\Bigg(2 T_{1,1}(t_0,t_1) \left(\big(|g|^{\frac12}(t_0)+|g|^{\frac12}(t_1)\big)^2-|g(t_0)-g(t_1)|\right)+|g(t_0)-g(t_1)| |g|^{\frac14}(t_0)|g|^{\frac14}(t_1)\\
&\,\,\,\,-4 |g|^{\frac12}(t_0) |g|(t_1)-4 |g|^{\frac12}(t_1) |g|(t_0) \Bigg)(\alpha(t_0,t_1)+\alpha(t_1,t_0)),
\end{aligned}
\end{equation*}
where $\alpha(t_0,t_1)$ is a matrix valued function given by
\begin{equation*}
\begin{aligned}
\alpha(t_0,t_1)=&|g|^{\frac12}(t_0)\Big(|g(t_0)-g(t_1)|-|g|(t_0)-|g|(t_1)+|g|(t_0) g^{-1}(t_0)g(t_1)\Big)g^{-1}(t_0)\\
&+|g|^{\frac12}(t_1) |g|(t_0) g^{-1}(t_0).
\end{aligned}
\end{equation*}
Using the Cayley–Hamilton theorem and the identity $\det(A)A^{-1}=\tr(A)-A$ which holds for every two by two matrix $A$, we have 
\begin{equation*}
\begin{aligned}
\Big(|g(t_0)-g(t_1)|-|g|(t_0)-|g|(t_1)+|g|(t_0) g^{-1}(t_0)g(t_1)\Big)g^{-1}(t_0)=&-|g|(t_1) g^{-1}(t_1).
\end{aligned}
\end{equation*}
This implies that 
$$\alpha(t_0,t_1)=|g|^{\frac12}(t_1) |g|(t_0) g^{-1}(t_0)-|g|^{\frac12}(t_0) |g|(t_1) g^{-1}(t_1).$$ 
Hence $\alpha$  is an anti-symmetric function and  the functions $F^{ij}_S$ vanish for all $i$ and $j$, in dimension two. 
This completes the proof of the theorem.
\end{proof}

\addcontentsline{toc}{section}{Summary and outlook}
\section*{Summary and outlook}

In this paper we introduced a new family of metrics on noncommutative tori, called functional metrics, and studied their spectral geometry.
We defined the Laplacian  of  these metrics and computed the heat trace asymptotics of these Laplacians.
A formula for the second density of the heat trace is obtained. In fact our formula covers a class larger than these Laplacians. We call the latter class, introduced in this paper,  Laplace type $h$-differential   operators.
As a result, the scalar curvature density and the total scalar curvature are explicitly computed in all dimensions for certain classes of functional metrics that include conformally flat
 metrics and twisted product of flat metrics. In dimension  two  our computations cover the total scalar curvature of all functional metrics.
Finally a Gauss-Bonnet type theorem for a noncommutative two torus equipped with a general functional metric is proved.
Extending noncommutative curvature computations beyond dimensions 2, 3, and 4, to all dimensions, and beyond conformally flat metrics has been an open problem that is effectively addressed in this paper for the first time. The moduli space of noncommutative metrics, even for  noncommutative two tori, is poorly understood at present. It is thus important to treat  wider classes of metrics by  heat equation and spectral geometry techniques in search of finding common patterns. 

\addcontentsline{toc}{section}{References}
\providecommand{\bysame}{\leavevmode\hbox to3em{\hrulefill}\thinspace}
\providecommand{\MR}{\relax\ifhmode\unskip\space\fi MR }
\providecommand{\MRhref}[2]{%
  \href{http://www.ams.org/mathscinet-getitem?mr=#1}{#2}
}
\providecommand{\href}[2]{#2}

\Addresses
\end{document}